\documentclass[11pt,a4paper, reqno]{amsart}
\usepackage{amsmath,amssymb}
\usepackage{bm}
\usepackage[dvipdfmx]{graphicx}
\usepackage{float}
\usepackage[T1]{fontenc}
\usepackage{textcomp}
\usepackage{type1cm}
\usepackage{bbm} 
\usepackage{pifont} 
\usepackage{amsthm} 
\usepackage{stmaryrd} 
\usepackage{cite} 
\usepackage{ascmac}
\usepackage[all]{xy}
\usepackage{tikz}
 \usetikzlibrary{arrows} 
\usepackage{mathrsfs} 
\usepackage{url}
\usepackage{alltt} 
\usepackage{braket} 
\usepackage{cases} 
\usepackage{multirow}

\usepackage[
  top=23mm,
  bottom=23mm,
  left=24mm,
  right=24mm
]{geometry}

\renewcommand{\labelenumi}{(\arabic{enumi})}

\usepackage{mathtools}
\mathtoolsset{showonlyrefs,showmanualtags}  

\makeatletter
\@addtoreset{equation}{section}

\makeatother

\usepackage[OT2,T1]{fontenc}
\DeclareSymbolFont{cyrletters}{OT2}{wncyr}{m}{n}
\DeclareMathSymbol{\Sha}{\mathalpha}{cyrletters}{"58} 

\theoremstyle{plain}
\newtheorem{thm}[equation]{Theorem}
\newtheorem{cor}[equation]{Corollary}

\newtheorem{prop}[equation]{Proposition}
\newtheorem{lem}[equation]{Lemma}

\theoremstyle{definition}
\newtheorem{dfn}[equation]{Definition}
\newtheorem{ex}[equation]{Example}
\newtheorem{rem}[equation]{Remark}

\newcommand{\rarrow}{\longrightarrow}

\usetikzlibrary{cd, decorations.pathmorphing}

\renewcommand{\hat}[1]{\widehat{#1}}

\newcommand{\fr}[1]{\mathfrak{#1}}

\newcommand{\cal}[1]{\mathcal{#1}}

\newcommand{\mapres}[2]{\left. #1 \vphantom{\big|} \right|_{#2}}

\DeclareMathOperator{\Hom}{Hom}

\let\Im\relax
\DeclareMathOperator{\Im}{Im}
\DeclareMathOperator{\Ker}{Ker}
\DeclareMathOperator{\Coker}{Coker}
\let\det\relax
\DeclareMathOperator{\det}{det}
\DeclareMathOperator{\Gal}{Gal}
\DeclareMathOperator{\Tr}{Tr}

\newcommand{\A}{\mathbb{A}}

\newcommand{\Z}{\mathbb{Z}}
\newcommand{\Q}{\mathbb{Q}}
\newcommand{\R}{\mathbb{R}}

\DeclareMathOperator{\ord}{ord}

\DeclareMathOperator{\norm}{Nr}

\newcommand{\Frob}{\mathrm{Frob}}

\DeclareMathOperator{\Ext}{Ext}
\DeclareMathOperator{\Extrig}{Extrig}
\DeclareMathOperator{\MExt}{MExt}
\DeclareMathOperator{\loc}{loc}

\DeclareMathOperator{\dR}{dR}
\DeclareMathOperator{\st}{st}
\DeclareMathOperator{\cris}{cris}
\DeclareMathOperator{\ur}{ur}
\newcommand{\BK}{\mathrm{BK}}
\newcommand{\Fil}{\mathrm{Fil}}
\newcommand{\dcris}{D_{\mathrm{cris}}}
\newcommand{\ddr}{D_{\mathrm{dR}}}
\newcommand{\dst}{D_{\mathrm{st}}}
\newcommand{\bklog}{\log^{\mathrm{BK}}}

\newcommand{\MF}{\mathbf{MF}}
\newcommand{\rep}{\mathop{\bm{\mathrm{Rep}}}}

\title[A difference formula of $p$-adic height pairings]{A difference formula of $p$-adic height pairings\\ via the Bloch--Kato logarithm map}
\author{Taiga Adachi}
\address{Joint Graduate School of Mathematics for Innovation, Kyushu University, Motooka 744, Nishi-ku Fukuoka 819-0395, Japan.}
\email{t.adachi1729@gmail.com}
\author{Yu Katagiri}
\address{Institute of Mathematics for Industry, Kyushu University, Motooka 744, Nishi-ku Fukuoka 819-0395, Japan.}
\email{yu.katagiri.s3@gmail.com}
\author{Ryota Shii}
\address{Graduate School for Mathematics, Kyushu University, Motooka 744, Nishi-ku Fukuoka 819-0395, Japan.}
\email{shii.ryota@gmail.com}


\begin{document}
\begin{abstract}
The construction of a $p$-adic height pairing for a geometric $p$-adic representation of the absolute Galois group of a number field depends on a global $p$-adic logarithm and on local splittings of the Hodge filtrations at the primes above $p$. 
We study the dependence on these splittings for suitable two-dimensional symplectic self-dual representations, including self-dual twists of representations attached to even-weight newforms at non-ordinary primes not dividing the level. 
We express the difference between the height pairings associated with the two splittings determined by Frobenius explicitly in terms of local Bloch--Kato logarithms. 
As an application over $\Q$, we prove that at least one of the two cyclotomic $p$-adic height pairings is non-trivial under the additional assumptions that the Frobenius eigenvalues at $p$ are distinct and the localization map at $p$ from the Bloch--Kato Selmer group is non-zero.
\end{abstract}

\maketitle

\section{Introduction}
\noindent
For a fixed rational prime $p$, a $p$-adic height pairing is a $p$-adic analogue of the real-valued N\'eron--Tate pairing. 
Its construction, however, depends on a global $p$-adic logarithm and splittings of the Hodge filtration on $p$-adic de Rham cohomology. 
The N\'eron--Tate pairing is non-degenerate, whereas a $p$-adic height pairing may be degenerate for some choices of logarithm. 
For the cyclotomic logarithm, non-degeneracy is conjectured but remains unproved in general; the problem is comparable in difficulty to the Leopoldt conjecture. 
Even non-triviality, meaning that the pairing is not identically zero, remains open in general, including for elliptic curves over $\Q$ with good ordinary reduction at $p$. 
Establishing non-triviality is thus an important arithmetic problem.

The conjectural relation between $p$-adic heights and $p$-adic $L$-functions helps explain this dependence on auxiliary choices. 
One expects a cyclotomic $p$-adic analogue of the Beilinson--Bloch--Kato conjecture for the $p$-adic realizations of motives, and more generally for suitable geometric $p$-adic representations. 
Although the corresponding $p$-adic $L$-functions are themselves conjectural in general, the Mazur--Tate--Teitelbaum $p$-adic BSD conjecture \cite{MTT86} for elliptic curves over $\Q$ provides a representative example, just as the classical BSD conjecture is a special case of the Beilinson--Bloch--Kato conjecture. 
It relates the order of vanishing to the Mordell--Weil rank and predicts a leading-term formula involving a $p$-adic regulator, the Tate--Shafarevich group, and other arithmetic invariants. 
The construction of such regulators from $p$-adic height pairings has been one of the principal motivations for the development of the theory.

The $p$-adic $L$-functions in these conjectures are attached to the cyclotomic $\Z_p$-extension, but Iwasawa theory also studies $p$-adic $L$-functions along other extensions, such as anticyclotomic extensions. 
The corresponding global choice in a height pairing is reflected in its logarithm. 
Even with the cyclotomic extension fixed, there remains a local choice in the construction of $p$-adic $L$-functions. 
For an elliptic curve over $\Q$ with good reduction at $p$, this choice involves the two roots $\alpha$ and $\beta$ of the Frobenius polynomial at $p$. 
In the ordinary case, the unit root gives a distinguished choice. 
In the supersingular case, however, neither root is canonically preferred, and both resulting $p$-adic $L$-functions play an important role, as illustrated by Pollack's plus/minus theory \cite{Pol03}.

The relation with $p$-adic $L$-functions thus suggests considering two height pairings corresponding to $\alpha$ and $\beta$, with the global logarithm fixed. 
For elliptic curves, a formula expressing their difference in terms of the formal logarithm is known \cite{BKO24, Kob13}. 
This formula has applications to the non-triviality of $p$-adic height pairings and to Perrin-Riou's conjecture on Beilinson--Kato elements. 
It is therefore natural to ask whether an analogous formula holds beyond the case of elliptic curves.

The general theory of $p$-adic heights developed by Zarhin \cite{Zar90}, Nekov\'a\v{r} \cite{Nek93}, and others provides a framework for this question. 
It extends height pairings from abelian varieties to Selmer groups of suitable geometric $p$-adic representations of the absolute Galois group of a number field. 
The construction again depends on a global $p$-adic logarithm and a splitting of the Hodge filtration at each prime above $p$. 
In this framework, the choices associated with $\alpha$ and $\beta$ for elliptic curves correspond to Hodge splittings determined by the Frobenius eigenlines. 
Under suitable hypotheses, two splittings arise in the same way for more general two-dimensional representations. 
However, the corresponding difference formula has not previously been established for representations attached to higher-weight modular forms.

In this paper, we fix a global $p$-adic logarithm and compute the difference between the height pairings attached to these splittings for suitable two-dimensional symplectic self-dual representations of the absolute Galois group of a number field. 
Our formula expresses the difference as a sum of local terms involving products of Bloch--Kato logarithms. 
This generalizes the elliptic-curve formula and applies, in particular, to self-dual twists of representations attached to even-weight newforms at non-ordinary primes. 
As an application over $\Q$, we prove that at least one of the two height pairings defined using the cyclotomic logarithm is non-trivial under the hypotheses stated below, including the assumption that the localization map at $p$ from the Bloch--Kato Selmer group is non-zero.

We describe our main result in detail.
Let $F$ be a number field, $G_{F}$ the absolute Galois group of $F$, and $E$ a finite extension of $\Q_{p}$.
Let $V$ denote a finite-dimensional $E$-vector space with a continuous $E$-linear action of $G_{F}$.
Although the results of this paper are established for a general number field $F$, we assume $F = \Q$ throughout this introduction for simplicity.
For a rational prime $\ell$, let $V_{\ell}$ denote the restriction of $V$ to $G_{\Q_{\ell}}$. 
Let $B_{\dR}$, $B_{\st}$ and $B_{\cris}$ be Fontaine's period rings.
For a $p$-adic representation $W$ of $G_{\Q_{p}}$ and $? \in \{ \dR, \st, \cris \}$, put $D_{?}(W) \coloneq (B_{?} \otimes_{\Q_{p}} W)^{G_{\Q_{p}}}$.
For any rational prime $\ell$, we define 
\begin{align}
H^{1}_{f}(\Q_{\ell}, V) &\coloneq \begin{cases}
  \Ker \left( H^{1}(\Q_{\ell}, V) \to H^{1}(I_{\ell}, V) \right) & (\ell \neq p), \\
  \Ker \left( H^{1}(\Q_{p}, V) \to H^{1}(\Q_{p}, V \otimes_{\Q_{p}} B_{\cris}) \right) & (\ell = p),
\end{cases}
\end{align}
and define the Bloch--Kato Selmer group for $V$ as 
\begin{align}
H^{1}_{f}(\Q, V) &\coloneq \Ker\left( H^{1}(\Q, V) \to \prod_{\ell} \frac{H^{1}(\Q_{\ell}, V)}{H^{1}_{f}(\Q_{\ell}, V)} \right),
\end{align}
where $I_{\ell}$ is the inertia group at $\ell$.
We suppose the following:
{
  \setlength{\leftmargini}{48pt}
  \begin{enumerate}
    \renewcommand{\labelenumi}{\textbf{(Hyp \arabic{enumi})}}
    \item The representation $V$ is unramified outside finitely many primes.
    \item The representation $V_{p}$ is semistable at $p$.
    \item For any rational prime $\ell$, we have 
    \begin{align}
        \begin{cases}
            H^{i}(\Q_{\ell},V_{\ell})=H^{i}(\Q_{\ell},V^{*}(1)_{\ell})=0 & (\ell \neq p,~ i=0,1,2),\\
            D_{\cris}(V_p)^{\varphi=1}=D_{\cris}(V^{*}(1)_{p})^{\varphi=1}=0 & (\ell = p).
        \end{cases}
    \end{align}
    Here, $V^{*} \coloneq \Hom_{E}(V, E)$, $W(1) \coloneq W \otimes_{E} E(1)$ is the Tate twist of a $p$-adic representation $W$, and $\varphi$ is the Frobenius on $D_{\mathrm{cris}}(V_{p})$.  
  \end{enumerate}
}
Let $l_{\Q}^{c}:\A_{\Q}^{\times}/\Q^{\times} \to \Q_{p}$ be the cyclotomic logarithm defined as 
\begin{align}
  l_{\Q, \ell}^{c}(x) \coloneq \begin{cases}
    \mathrm{ord}_{\ell}(x)\log_{p}\ell & (\ell \nmid p\infty), \\
    -\log_{p}x & (\ell = p), \\
    0 & (\ell = \infty),
  \end{cases}
\end{align}
where $\mathrm{ord}_{\ell}$ is the normalized valuation at $\ell$, and $\log_{p}$ is the $p$-adic logarithm on $\mathbb{Q}_{p}^{\times}$ satisfying $\log_{p}p=0$.
If we take an $E$-linear splitting $N:D_{\mathrm{dR}}(V_{p})/D_{\mathrm{dR}}^{+}(V_{p}) \to D_{\mathrm{dR}}(V_{p})$ of the Hodge filtration 
\begin{align}
  \xymatrix{
    0 \ar[r] & D_{\mathrm{dR}}^{+}(V_{p}) \ar[r] & D_{\mathrm{dR}}(V_{p}) \ar[r] & D_{\mathrm{dR}}(V_{p})/D_{\mathrm{dR}}^{+}(V_{p}) \ar[r] & 0,
  }
\end{align}
where $D_{\dR}^{+}(V_{p}) \coloneq (B_{\dR}^{+} \otimes_{\Q_{p}} V_{p})^{G_{\Q_{p}}}$, we can define a $p$-adic height pairing
\begin{align}
  \langle \phantom{a}, \phantom{a} \rangle_{l_{\Q}^{c}, N} = \langle \phantom{a}, \phantom{a} \rangle_{N}:H^{1}_{f}(\Q, V^{*}(1)) \times H^{1}_{f}(\Q, V) \to E
\end{align}
following Nekov\'a\v{r} \cite{Nek93}.
We also assume that $V$ is a two-dimensional symplectic self-dual representation; namely, that there exists a non-degenerate, skew-symmetric, and $G_{\Q}$-equivariant pairing
\begin{align}
  [\phantom{a}, \phantom{a}]:V \times V \to E(1).
\end{align}
Since $V_p$ is a two-dimensional Hodge--Tate representation and, by self-duality, its two Hodge–Tate weights sum to $1$, it follows that $\dim_ED_{\dR}^{+}(V_{p})=1$.
Now, we fix an $E$-basis $\omega$ of $D_{\dR}^{+}(V_{p})$.
Furthermore, suppose that $V$ satisfies
{
  \setlength{\leftmargini}{48pt}
  \begin{enumerate}
    \setcounter{enumi}{3}
    \renewcommand{\labelenumi}{\textbf{(Hyp \arabic{enumi})}}
    \item $\{ \varphi \omega, \omega \}$ forms an $E$-basis of $D_{\dR}(V_{p})$.
\end{enumerate}
}
Note that if $A$ is an elliptic curve defined over $\Q$ with good reduction at $p$ and we put $V=V_{p}A$, this assumption \textbf{(Hyp 4)} is equivalent to the statement that $A$ has no $p$-ordinary CM by a result of Serre--Tate (cf. \cite{Ser89}).
Let $\alpha, \beta \in \overline{E}$ be the eigenvalues of $\varphi$ on $D_{\st}(V_{p})$.
Let $L$ denote a finite extension of $E$ containing $\alpha$ and $\beta$.
For $\gamma \in \{ \alpha, \beta \}$, the $\gamma$-eigenspace of $\varphi$ can be seen as a splitting $N_{\gamma}:D_{\dR}(V_{p})_{L}/D_{\dR}^{+}(V_{p})_{L} \to D_{\dR}(V_{p})_{L}$ from the above hypotheses. 
Here, $X_L:=X\otimes_{E}L$ for an $E$-vector space $X$.
After extending scalars to $L$, we regard these height pairings as $L$-valued.
The following is our main theorem.

\begin{thm}[Theorem \ref{thm:MainTheorem}] \label{main_theorem1}
  We have
  \begin{align}
    \langle a, a \rangle_{N_{\alpha}} - \langle a, a \rangle_{N_{\beta}} 
    = \frac{\beta-\alpha}{[\varphi\omega, \omega]_{\dR}}\bklog_{V_{p}}(a_{p})(\omega)^2,
  \end{align}
  for $a \in H^{1}_{f}(\Q, V)$.
  Here, 
  \begin{itemize}
      \item $[\phantom{a}, \phantom{a}]_{\mathrm{dR}}: D_{\mathrm{dR}}(V_{p}) \times D_{\mathrm{dR}}(V_{p}) \to D_{\dR}(E(1)) = E$ is the de Rham pairing induced by the pairing $[\phantom{a}, \phantom{a}]$, 
      \item $\log_{V_{p}}^{\mathrm{BK}}:H^{1}_{f}(\Q_{p}, V_{p}) \to \Hom(D_{\mathrm{dR}}^{+}(V_{p})_{L}, D_{\dR}(E(1))_{L})$ is the Bloch--Kato logarithm map for $V_{p}$.
      \item $a_{p} \in H^{1}_{f}(\Q_{p}, V)$ is the image of $a$ by the localization map $H^{1}_{f}(\Q, V) \to H^{1}_{f}(\Q_{p}, V)$.
  \end{itemize}
\end{thm}
\begin{rem}
  \begin{enumerate}
    \item Since $D_{\mathrm{dR}}^{+}(V_{p})$ is orthogonal to itself under the de Rham pairing $[\phantom{a}, \phantom{a}]_{\dR}$, the hypothesis \textbf{(Hyp 4)} implies $[\varphi\omega, \omega]_{\dR} \neq 0$.
    \item By the hypotheses \textbf{(Hyp 2)} and \textbf{(Hyp 3)}, the Bloch--Kato logarithm map $\log_{V_{p}}^{\mathrm{BK}}$ is well-defined.
    See Section 2.1 for details.
    \item In Theorem \ref{thm:MainTheorem}, the global logarithm may be any non-trivial continuous homomorphism $l_{F}:\A_{F}^{\times}/F^{\times} \to \Q_{p}$.
    \item Let $A$ be an elliptic curve defined over $\Q$ with good supersingular reduction at $p$. Bernardi and Perrin-Riou \cite{Ber-Per93} constructed a $D_p(V_pA)$-valued $p$-adic height function on $A(\Q)$, and Kobayashi \cite[Theorem 4.8]{Kob13} identified its components with the diagonal restrictions of the Zarhin--Nekov\'{a}\v{r}-type height pairings. Taking the Frobenius eigenlines as Hodge splittings, one sees that Theorem \ref{main_theorem1} recovers the known difference formula for elliptic curves \cite{BKO24, Kob13}.
  \end{enumerate}
\end{rem}
As mentioned above, we also construct a framework for proving the non-triviality of $p$-adic height pairings for $p$-adic Galois representations. 

\begin{cor}[Corollary \ref{cor:non-triviality}] \label{cor:main_theorem}
  We assume that
  \begin{enumerate}
  \renewcommand{\labelenumi}{(\roman{enumi})}
    \item $\alpha \neq \beta$,
    \item the localization map $H^{1}_{f}(\Q, V) \to H^{1}_{f}(\Q_{p}, V_{p})$ at $p$ is a non-zero map.
  \end{enumerate}
  Then, either $\langle \phantom{a}, \phantom{a} \rangle_{N_{\alpha}}$ or $\langle \phantom{a}, \phantom{a} \rangle_{N_{\beta}}$ is non-trivial.
  In particular, if the characteristic polynomial of $\varphi$ on $D_{\st}(V_{p})$ is irreducible over $E$, both $\langle \phantom{a}, \phantom{a} \rangle_{N_{\alpha}}$ and $\langle \phantom{a}, \phantom{a} \rangle_{N_{\beta}}$ are non-trivial.
\end{cor}
\begin{rem}
  \begin{enumerate}
    \item If an elliptic curve $A$ has positive Mordell--Weil rank and we put $V = V_p A$, the hypothesis (ii) in Corollary \ref{cor:main_theorem} holds by $A(\Q) \subset A(\Q_{p})$.
    By the work by Burungale--Skinner--Wan \cite{BSW26+}, the hypothesis (ii) in Corollary \ref{cor:main_theorem} also holds when $V$ is a $p$-adic representation associated with a weight 2 newform $f$ satisfying $\ord_{s=1}L(f, s) = 1$. 
    \item For weight-two $p$-ordinary newforms of analytic rank one, B\"{u}y\"{u}kboduk--Pollack--Sasaki \cite{BPS21+} proved that at least one of the two $p$-adic height pairings is non-trivial. 
    Corollary \ref{cor:main_theorem} gives an alternative proof of this result without using $p$-adic Gross--Zagier formulae, whereas the proof in \cite{BPS21+}, following the strategy of B\"{u}y\"{u}kboduk \cite{Buy20}, relies on  those established in \cite{Kob, BN26+}.
    \item Unlike Theorem \ref{main_theorem1}, it is essential to take the cyclotomic logarithm map in Corollary \ref{cor:main_theorem}.
    For details, see the proof of Corollary \ref{cor:non-triviality}.
  \end{enumerate}
\end{rem}
As discussed above, for an arithmetic application of Theorem \ref{main_theorem1} and Corollary \ref{cor:main_theorem}, we will consider the Galois representation associated to an eigen-newform of weight $k \geq 2$ and obtain a generalization of \cite[Corollary A.3]{BKO24}.
Let $f = \sum_{n \geq 1} a_{n}(f)q^{n}$ be a normalized eigen-newform of even weight $k \geq 2$ with level $\Gamma_{0}(N)$.
Fix a place $\mathfrak{p} \mid p$ of the Hecke field $K_{f} \coloneq \mathbb{Q}(a_{n}(f) ~\vline~ n \geq 1)$ of $f$. 
Let $V_{f, \mathfrak{p}}$ be the two-dimensional $p$-adic Galois representation of $G_{\Q}$ over $K_{f, \mathfrak{p}}$ associated to $f$ by Deligne.
Here, $K_{f, \mathfrak{p}}$ is the completion of $K_{f}$ at $\mathfrak{p}$.
Assume that $p \nmid N$, that is, assume that $V_{f, \fr{p}}$ is a crystalline representation.
We put $V_{\fr{p}}\coloneq V_{f, \fr{p}}(k/2)$.
Let $\omega_{f} \in D_{\dR}(V_{\fr{p}})$ be the element corresponding to the algebraic differential form attached to $f$ by the comparison theorem.
Applying Theorem \ref{main_theorem1} and Corollary \ref{cor:main_theorem} to $V_{\fr{p}}$, we obtain the following corollary.

\begin{cor}[Corollary \ref{cor:differ_modular}]\label{cor:arith_application}
  Suppose that $f$ is non-ordinary at $p$, namely, that $a_{p}(f)$ is not a $p$-adic unit in $K_{f, \fr{p}}$.
  Let $\alpha$ and $\beta$ be the roots of the Hecke polynomial $X^{2} -a_{p}(f)p^{-k/2}X + p^{-1}$ at $p$.
  Then, we have 
  \begin{align}
    \langle a, a \rangle_{N_{\alpha}} - \langle a, a \rangle_{N_{\beta}} = \frac{\beta-\alpha}{[\varphi\omega_{f}, \omega_{f}]_{\dR}} \log^{\BK}_{V_{\fr{p}}}(a_p)(\omega_f)^{2}
  \end{align}  
  for $a \in H^{1}_{f}(\Q, V_{\fr{p}})$.
  Furthermore, if the localization map $H^{1}_{f}(\Q, V_{\fr{p}}) \to H^{1}_{f}(\Q_{p}, V_{\fr{p}})$ at $p$ is non-zero and $\alpha \neq \beta$, then either the pairing $\langle \phantom{a}, \phantom{a} \rangle_{N_{\alpha}}$ or the pairing $\langle \phantom{a}, \phantom{a} \rangle_{N_{\beta}}$ attached to $V_{\fr{p}}$ is non-trivial. 
\end{cor}

\subsection*{Organization of the paper}
In Section 2, we recall the Zarhin--Nekov\'a\v{r}-type construction of the $p$-adic height pairing for $p$-adic representations.
In Section 3, we calculate the difference between the $p$-adic height pairings attached to two splittings, and give a proof of Theorem \ref{main_theorem1} and Corollary \ref{cor:main_theorem}.
In Section 4, we obtain Corollary \ref{cor:arith_application} by applying Theorem \ref{main_theorem1} and Corollary \ref{cor:main_theorem} to the representation attached to a normalized Hecke eigen-newform.

\subsection*{Notation}
\begin{itemize}
  \item Throughout this paper, let $p$ be a rational prime.
  \item For a perfect field $k$, let $G_k$ be its absolute Galois group.
  \item For a finite extension $K/\Q_p$, let $\mathcal{O}_K$ be its integer ring.
  \item For a number field $F$ and a place $v$ of $F$, let $F_v$ be the completion of $F$ at $v$. 
  If $v$ is non-archimedean, let $\ord_v:F_v^{\times}\to \Z$ be the valuation normalized by $\ord_v(\pi)=1$ for a fixed uniformizer $\pi$ of $F_v$.
  \item For a perfect field $k$ and a finite extension $E/\Q_{p}$, a $p$-adic Galois representation of $G_k$ over $E$ is a finite-dimensional $E$-vector space $V$ with a continuous $E$-linear $G_{k}$-action.
  Let $\rep_{G_k, E}$ be the category of $p$-adic Galois representations of $G_k$ over $E$. 
  \item For a number field $F$, a finite place $v$ of $F$, and an object $X$ of $\rep_{G_{F}, E}$, let $X_{v}$ be the $p$-adic Galois representation whose $G_{F_{v}}$-action is obtained by the restriction of $G_{F}$ to $G_{F_{v}}$.
  \item For a perfect field $k$ and an object $X$ of $\rep_{G_k, E}$, we put $X^* \coloneq \Hom_{E}(X, E)$.
  \item For a perfect field $k$ and an object $X$ of $\rep_{G_k, E}$ and $i \in \Z$, we denote the $i$-th Tate twist of $X$ by $X(i)$.
  \item For a finite-dimensional $E$-vector space $X$ and a finite extension $L/E$,  we put $X_L \coloneq X \otimes_{E} L$.
  For an $E$-linear map $f:X \to Y$ of finite-dimensional $E$-vector spaces, we let $f_L:X_L\to Y_L$ be the $L$-linear map induced by $f$.
  We also denote the map sending $x \in X$ to $f(x)\otimes 1 \in Y_L$ by $f: X \to Y_L$ if there is no confusion.
  \item For a finite extension $K/\Q_{p}$ and an object $X$ of $\rep_{G_K, E}$, let $\varphi$ be the Frobenius of $\dst(X)$ and $\dcris(X)$, and let $N$ be the monodromy operator of $\dst(X)$. 
  \item For $x\in H^1(G_{F},X)$, we denote the image of $x$ under the localization map at $v$ by $x_v$.
\end{itemize}

\subsection*{Acknowledgements}
The authors would like to express their sincere gratitude to Shinichi Kobayashi for introducing this topic, for his consistent encouragement and fruitful discussions, and for the insights gained through his papers \cite{Kob13} and \cite{Kob14}. 
They are grateful to Kentaro Nakamura for providing many significant comments, reading an earlier version of the manuscript carefully, and pointing out mathematical mistakes.
They would also like to thank K\^{a}z\i m B\"{u}y\"{u}kboduk, Kazuto Ota, Chan-Ho Kim, Naoto Dainobu, and Shoma Sueyoshi for their helpful comments and discussions. 
The first author was supported by WISE program (MEXT) at Kyushu University. 

\section{$p$-adic height pairings for $p$-adic Galois representations}\label{sec:p-adic_height}
\noindent
In this section, we give an overview of the Zarhin--Nekov\'{a}\v{r}-type construction of $p$-adic height pairings. 
Nekov\'{a}\v{r} constructed global $p$-adic height pairings and showed that they coincide with a sum of local $p$-adic height pairings in \cite[\S2-\S4]{Nek93}. 
In this paper, following Nekov\'{a}\v{r}'s results, we define a global $p$-adic height pairing as a sum of local $p$-adic height pairings, which will be convenient for subsequent calculations.
The main reference is \cite[\S 2-4, \S 7]{Nek93}, but he occasionally omitted some details. 
In this section, we supply them for readers' convenience.

\subsection{Setup}\label{subsec:set-up}
Let $F$ be a number field. 
Let $V$ be an object of $\rep_{G_F, E}$. 
For each finite place $v$ of $F$, let $I_{v}$ be the inertia subgroup. 
For $v\mid p$, we also write $F_{v,0}$ for the maximal unramified extension of $\Q_{p}$ inside $F_{v}$.

In this paper, we assume that
{
  \setlength{\leftmargini}{48pt}
  \begin{enumerate}
    \renewcommand{\labelenumi}{\textbf{(Hyp \arabic{enumi})}}
    \item There exists a finite set $S_0$ of places of $F$ such that $V$ is unramified outside $S_0$.
    \item For every place $v$ above $p$, $V_v$ is a semistable representation of $G_{F_v}$.
    \item For every finite place $v$, $L_v(V,0)L_v(V^*(1),0)\neq 0$, where 
    \begin{align}
      L_v(V, 0) \coloneq \begin{cases}
        \det(1-\Frob_v \ |\ V_v^{I_v}) & (v \nmid p), \\
        \det(1-\varphi^{[F_{v,0}:\Q_p]} \ |\ \dcris(V_v)) & (v \mid p).
      \end{cases}
    \end{align}
    Here, $\Frob_v$ is a representative of the geometric Frobenius element of $G_{F_v}$. 
  \end{enumerate}
}

\textbf{(Hyp 3)} implies that 
\begin{align}
  \label{eq:Galois_v}
  &\bullet\ H^i(F_v,V_v)=H^i(F_v,V^*(1)_v)=0\ \ \ (v\nmid p,\ i=0,1,2),\\
  \label{eq:cris_v}
  &\bullet\ \dcris(V_v)^{\varphi=1}=\dcris(V^*(1)_v)^{\varphi=1}=0\ \ \ (v\mid p).
\end{align}

For each finite place $v$ and $* \in \{\mathrm{cris},\mathrm{st}\}$, we define
\begin{align}
  &H^1_{\mathrm{ur}}(F_v,V_v) \coloneq \Ker(H^1(G_{F_v},V_v)\to H^1(I_v,V_v))\quad (v\nmid p),\\
  &H^1_*(F_v,V_v) \coloneq \Ker(H^1(G_{F_v},V_v)\to H^1(G_{F_v},V_v\otimes_{\Q_p}B_*))\quad (v\mid p).
\end{align}
We put $H^1_f \coloneq H^1_{\mathrm{cris}}$. 
Furthermore, for $*\in\{f,\st\}$ and a finite set $S$ of places of $F$ containing $S_0\cup\{v\mid p\}$, we define the Selmer groups 
\begin{align}
  H^1_*(G_{F,S},V) \coloneq \Ker\left(H^1(G_{F,S},V)\to \prod_{\substack{v\in S\\ v\nmid p}}\frac{H^1(F_v,V_v)}{H^1_{\mathrm{ur}}(F_v,V_v)}\times \prod_{\substack{v\in S\\ v\mid p}}\frac{H^1(F_v,V_v)}{H^1_*(F_v,V_v)}\right),
\end{align}
where $G_{F, S}$ is the Galois group of the maximal extension $F_S$ unramified outside $S$ over $F$.
Note that we can ignore archimedean places, since $H^1(\R,V_v)=0$ for every real place $v$ of $F$ (even when $p=2$). 
By the hypothesis \textbf{(Hyp 1)} and \eqref{eq:Galois_v}, $H^1(G_{F,S},V)$ does not depend on the choice of $S$ by the standard argument (see, for example, \cite[Proof of Lemma 4]{Jan89}) and coincides with $H^1(F,V)$. 
The same is true for $H^1_*(G_{F,S},V)$ for $*\in \{\mathrm{cris},\st\}$.

Under \textbf{(Hyp 2)}, we have the canonical exact sequence, for $v\mid p$,
\begin{align}
0\to H^1_f(F_v,V_v)\to H^1_{\mathrm{st}}(F_v,V_v)\to(\dcris(V^*(1)_v)^*)^{\varphi=1}\to 0
\end{align} 
by \cite[Corollary 1.18]{Nek93}. 
It follows from \eqref{eq:cris_v} that $H^1_{\mathrm{st}}(F_v,V_v)=H^1_f(F_v,V_v)$ for every $v\mid p$, and hence $H^1_{\mathrm{st}}(G_{F,S},V)=H^1_f(G_{F,S},V)$.

We define the Bloch--Kato logarithm map (cf.\cite[\S3]{BK90}). 
Fix a place $v$ of $F$ above $p$. 
The fundamental exact sequence 
\begin{align}
  \xymatrix{
    0 \ar[r] & \Q_p\ar[r] & B_{\mathrm{cris}}^{\varphi=1}\ar[r] & B_{\dR}/B_{\dR}^+\ar[r] & 0
  }  
\end{align}
induces a natural surjective map
\begin{align}
  \exp_{V_v}^{\mathrm{BK}}:\ddr(V_v)/\ddr^+(V_v)\to H^1_e(F_v,V_v),
\end{align}
where
\begin{align}
  H^1_e(F_v,V_v) \coloneq \Ker(H^1(F_v,V_v)\to H^1(F_v,V_v\otimes_{\Q_p} B_{\mathrm{cris}}^{\varphi=1})).
\end{align}
We call $\exp_{V_v}^{\mathrm{BK}}$ the \textit{Bloch--Kato exponential map} for $V_v$. 
Note that $H^1_e(F_v,V_v)=H^1_f(F_v,V_v)$ by \eqref{eq:cris_v} and \cite[Corollary 1.16]{Nek93}. 
We also see that $\exp_{V_v}^{\mathrm{BK}}$ is an $E$-linear isomorphism, since the fundamental sequence and \eqref{eq:cris_v} imply that $\Ker(\exp_{V_v}^{\mathrm{BK}})=\dcris(V_v)^{\varphi=1}/H^0(F_v,V_v)=0$.
Then the \textit{Bloch--Kato logarithm map} $\bklog_{V_v}$ for $V_v$ is defined by
\begin{align}
  \bklog_{V_v}:H^1_{f}(F_v,V_v)\overset{(\exp_{V_v}^{\mathrm{BK}})^{-1}}{\longrightarrow}\ddr(V_v)/\ddr^+(V_v)\cong \Hom_{F_{v} \otimes_{\Q_{p}} E}(\ddr^+(V^*(1)_v), D_{\dR}(E(1))). \\
  \label{definition:Bloch-Kato_log}
\end{align}
In this paper, for a $p$-adic representation $X$ satisfying similar hypotheses \textbf{(Hyp 1)}-\textbf{(Hyp 3)} in \S \ref{subsec:set-up}, we identify $H^1_{\st}(F_v,X_v)$ with the tangent space $\ddr(X_v)/\ddr^+(X_v)$ via the Bloch--Kato exponential map whenever necessary, for a place $v$ of $F$ above $p$.

\subsection{Mixed extensions}
For $p$-adic Galois representations $X$ and $Y$ of $G_F$, let $\Ext^1_F(X,Y)$ be the set of isomorphism classes of extensions.
Hence an element of $\Ext^1_F(X,Y)$ is represented by a short exact sequence
\begin{align}
  \xymatrix{
    0\ar[r] & Y\ar[r] & \cal{E} \ar[r] &X\ar[r] &0
  }
\end{align}
in $\rep_{G_F, E}$. 
We denote its isomorphism class by $[\cal{E}]$. 
The set $\Ext^1_F(X,Y)$ has a structure of an abelian group by the \textit{Baer sum} $\wedge$ defined as follows: for $[\cal{E}_1],[\cal{E}_2] \in \Ext^1_F(X,Y)$ represented by
\begin{align}
  \xymatrix{
    0 \ar[r] & Y\ar[r] & \cal{E}_1\ar[r]^{\pi_1} & X \ar[r] & 0, \quad 0 \ar[r] & Y \ar[r] & \cal{E}_2 \ar[r]^{\pi_2} & X \ar[r] & 0
  }  
\end{align}
respectively, $\cal{E}_1\wedge \cal{E}_2$ is defined by
\begin{align}\label{def_Baer_sum}
  \cal{E}_1 \wedge \cal{E}_2 \coloneq \frac{\left\{(u_1,u_2)\in \cal{E}_1 \oplus \cal{E}_2\ \middle|\ \pi_1(u_1) = \pi_2(u_2)\right\}}{\left\{(u_1,u_2)\ \middle|\ u_1,u_2\in Y,\ u_1+u_2=0\right\}},
\end{align}
and the corresponding extension is defined by
\begin{equation}\label{Baer_sum_extension}
  \xymatrix{
  0\ar[r] & Y\ar[r]^-{\iota_0} & \cal{E}_1\wedge \cal{E}_2\ar[r]^{\ \ \ \pi_1(=\pi_2)}& X\ar[r]&0,
  }
\end{equation}
where $\iota_0(y) \coloneq (y,0)$. 
Moreover, we define a structure of an $E$-vector space for $\Ext^1_F(X, Y)$ by the scalar multiplication
\begin{align}\label{extension_scalar}
  a\cdot[\cal{E}] \coloneq \begin{cases}
    0\to Y\longrightarrow \cal{E}\overset{a^{-1}\pi}{\longrightarrow} X\to 0 & (a\neq0), \\
    [X\oplus Y] & (a=0),
  \end{cases}
\end{align}
for $a \in E$. 
For semistable (resp. crystalline) repsentations $X$ and $Y$, we denote by $\Ext^1_{F,\mathrm{st}}(X,Y)$ (resp. $\Ext^1_{F,\mathrm{cris}}(X,Y)$) the $E$-subspace of $\Ext^1_F(X,Y)$ consisting of isomorphism classes of extensions $[\cal{E}]$ for which the class of $\cal{E}_v$ lies in $H^1_{\ur}(F_v,\Hom_E(X,Y))$ for $v \nmid p$, and $\cal{E}_v$ is a semistable (resp. crystalline) representation of $G_{F_v}$ for $v \mid p$. 
It follows that if $X$ is a trivial representation, there exist isomorphisms
\begin{align}\label{isom:Ext_Hom}
  \Ext^1_F(X,Y)\cong \Hom(X,H^1(F,Y)),\quad \Ext^1_{F,*}(X,Y)\cong \Hom(X,H^1_*(F,Y))
\end{align}
for $*\in\{\mathrm{cris},\st\}$ given by sending each extension to the connecting map. 

In this section, let $A$ and $B$ be trivial $p$-adic Galois representations of $G_F$. 
We recall the definition of mixed extensions to define local $p$-adic height pairings. 
Let $e_1=[\cal{E}_1]\in\Ext^1_{F}(A,V)$ and $e_2=[\cal{E}_2]\in\Ext^1_{F}(V,B(1))$ be extensions. 
Throughout the paper, we assume that $\cal{E}_{1,v}$ and $\cal{E}_{2,v}$ are de Rham representations for all $v \mid p$.
\begin{dfn}[{cf. \cite[Definition 4.3]{Nek93}}]
  \begin{itemize}
    \item A \textit{mixed extension} of $e_1$ and $e_2$ is a $p$-adic Galois representation $\cal{E}$ with an increasing filtration of subrepresentations
    \begin{align}
      0=W_{-3}\cal{E} \subseteq W_{-2}\cal{E} \subseteq W_{-1}\cal{E} \subseteq W_0\cal{E} = \cal{E}
    \end{align}
    such that there exist two isomorphisms $W_{-1}\cal{E} \cong \cal{E}_2$ and $W_0\cal{E}/W_{-2}\cal{E}\cong \cal{E}_1$. 
    This is equivalent to the condition that there exist the morphisms $f_1:\cal{E}_2\to \cal{E}$, $f_2:B(1) \to \cal{E}$, $f_3:\cal{E} \to A$, and $f_4:\cal{E} \to \cal{E}_1$ for which the diagram of Figure \ref{def_mixed_extension} with exact rows and columns is commutative.
    \begin{figure}[H]
      \begin{align}
        \xymatrix{
          & & 0 \ar[d] & 0 \ar[d] & \\
          0 \ar[r] & B(1) \ar@{=}[d] \ar[r] & \cal{E}_2\ar^{f_1}[d] \ar[r] & V \ar[r] \ar[d] & 0 \\
          0 \ar[r] & B(1) \ar^{f_2}[r] & \cal{E} \ar^{f_4}[r] \ar^{f_3}[d] & \cal{E}_1 \ar[r] \ar[d] & 0 \\
          & & A \ar@{=}[r] \ar[d] & A \ar[d] \\
          & & 0 & 0 & 
        } \label{eq:def_mixed}
      \end{align}
    \caption{A mixed extension}\label{def_mixed_extension}
    \label{Definition:Mixed_extensions}
    \end{figure}

    \item Let $\cal{E}$, $\cal{E}'$  be mixed extensions of $e_1$ and $e_2$. 
    A morphism from $\cal{E}$ to $\cal{E}'$ consists of $g:\cal{E} \to \cal{E}'$, $g_1:\cal{E}_1\to \cal{E}_1'$, and $g_2:\cal{E}_2 \to \cal{E}_2'$ such that the diagram of Figure \ref{diag_mor_mixed_extension} is commutative. 
    A morphism from $\cal{E}$ to $\cal{E}'$ is called an isomorphism of mixed extensions if $g:\cal{E} \to \cal{E}'$, $g_1:\cal{E}_1 \to \cal{E}_1'$, and $g_2:\cal{E}_2 \to \cal{E}_2'$ are isomorphisms.
    \begin{figure}[H]
      \centering
      \includegraphics{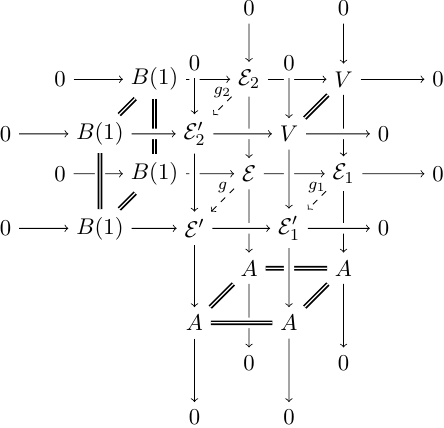}
      \caption{A morphism of mixed extensions}\label{diag_mor_mixed_extension}
      \label{diagram:mixed_extension}
    \end{figure}
    We denote by $\MExt(e_1,e_2)$ the set of isomorphism classes of mixed extensions of $e_1$ and $e_2$.
  \end{itemize}
\end{dfn}

We also introduce \textit{essentially unramified} mixed extensions to show well-definedness of the global $p$-adic height pairings (see \S\ref{sec:global_height}).

\begin{prop}\label{prop:l-adic_splitting}
  For $[\cal{E}]\in \MExt(e_1,e_2)$ and $v\nmid p$, there exists a unique $[\cal{C}_v]\in\Ext^1_{F_v}(A,B(1))$ satisfying $\cal{E}_v \cong V_v\oplus \cal{C}_v$ as mixed extensions of $E$-representations of $G_{F_{v}}$. 
\end{prop}

\begin{proof}
  By \eqref{eq:Galois_v}, we see that
  \begin{align}
    &\Ext^{i}_{F_v}(A,V_v)=\Hom(A,H^{i}(F_v,V_v))=0, \\
    &\Ext^{i}_{F_v}(V_v,B(1))\cong \Ext^{i}_{F_v}(B^{*}, V^{*}(1)_v)=\Hom(B^*,H^{i}(F_v,V^*(1)_v))=0
  \end{align}
  for $i = 0, 1$.
  It follows that we have $\cal{E}_{1,v}\cong A\oplus V_v$ and $\cal{E}_{2,v}\cong B(1)\oplus V_v$ as $G_{F_v}$-modules. 
  Let $j_{A}:A \to \cal{E}_{1, v}$ denote a $G_{F_{v}}$-section, which is unique since $\Ext_{F_{v}}^{0}(A, V_{v}) = 0$, and let $j_{V}:V_{v} \to \cal{E}_{1, v}$ be the inclusion.
  Then we have
  \begin{align}
    (j_{A}^{*}, j_{V}^{*}):\Ext^1_{F_v}(\cal{E}_{1,v},B(1)) \cong \Ext^1_{F_v}(A\oplus V_v,B(1))\cong\Ext^1_{F_v}(A,B(1))\times \Ext^1_{F_v}(V_v,B(1)),
  \end{align}
  where the inverse of the last isomorphism sends $([X],[Y])$ to $[X \oplus_{B(1)} Y]$.  
  Here $\oplus_{B(1)}$ means the fiber sum over $B(1)$.   
  Since we have $[\cal{E}_{v}] \in \Ext_{F_{v}}^{1}(\cal{E}_{1, v}, B(1))$ and the preimage of $j_V(V_v)$ under $\cal{E}_{v} \to \cal{E}_{1, v}$ is $\cal{E}_{2, v}$ by definition, we see that $j_{V}^{*}[\cal{E}_{v}]=[\cal{E}_{2, v}] = 0$.
  Putting $[\cal{C}_v] \coloneq j_{A}^{*}[\cal{E}_v]$, we obtain
  \begin{align}
    \cal{E}_v \cong \cal{C}_v\oplus_{B(1)}\bigl(V_v\oplus B(1)\bigr)\cong \cal{C}_v \oplus V_v
  \end{align}
  as mixed extensions.
  Conversely, let $\phi:V_{v} \oplus \cal{C}' \overset{\sim}{\to} \cal{E}_{v}$ be another isomorphism of mixed extensions. 
  Then $\phi$ descends to an isomorphism $\overline{\phi} : V_{v} \oplus A \overset{\sim}{\to} \mathcal{E}_{1, v}$ restricting to the inclusion on $V_{v}$ and to a $G_{F_{v}}$-section on $A$, whence $\overline{\phi} = j_{V} \oplus j_{A}$.
  Therefore, we obtain $[\cal{E}_{v}] = [V_{v} \oplus \cal{C}']$ in $\Ext_{F_{v}}^{1}(\cal{E}_{1, v}, B(1))$, and applying $j_{A}^{*}$ gives $[\cal{C}'] = [\cal{C}_{v}]$.
\end{proof}

\begin{dfn}[{\cite[Definition 4.3]{Nek93}}]
  For $[\cal{E}] \in \MExt(e_1,e_2)$ and $v\nmid p$, $\cal{E}$ is called \textit{essentially unramified} at $v$ if $\cal{C}_v$ in Proposition \ref{prop:l-adic_splitting} is unramified.
\end{dfn}

We define $\MExt_{\st}(e_1,e_2)$ as the set of elements $[\cal{E}]\in\MExt(e_1,e_2)$ satisfying
\begin{enumerate}
  \renewcommand{\labelenumi}{(\roman{enumi})}
  \item $\cal{E}$ is essentially unramified at all $v\nmid p$,
  \item $\cal{E}_v$ is semistable for all $v\mid p$.
\end{enumerate}

\begin{rem}
  We give the above definition inspired by the crystalline case in \cite[Proposition 4.4 (2)]{Nek93}. 
  However, the condition (ii) automatically follows from \cite[Corollary 1.27 (2) and Proposition 4.4 (4)]{Nek93}, and hence it can be omitted, in fact.
\end{rem}

For $[X]\in \Ext^1_F(A,B(1))$, let $[\overline{X}]$ be the image of $[X]$ under the morphism $\Ext^1_F(A,B(1))\to \Ext^1_F(A,E_2)$. 
We define the action
\begin{align}
  \Psi:\Ext^1_F(A,B(1))\times \MExt(e_1,e_2)\to \MExt(e_1,e_2)
\end{align}
by $\Psi([X],[\cal{E}]) \coloneq [\overline{X}]\wedge[\cal{E}]$, where $\wedge$ is the Baer sum of $\Ext_F^1(A,\cal{E}_2)$ defined in \eqref{def_Baer_sum}. 
It is known that this action $\Psi$ is well-defined. 
The following proposition describes the relationship between $\MExt(e_1,e_2)$ and $\Ext^1_F(A,B(1))$, which will be used in \S\ref{sec:differ_splitting}.

\begin{prop}[cf.{\cite[Proposition 4.4 (1), (2)]{Nek93}}]\label{prop:principal_mixed_extension}
  The sets $\MExt(e_1,e_2)$ and $\MExt_{\st}(e_1,e_2)$ are principal homogeneous spaces of $\Ext^1_{F}(A,B(1))$ and $\Ext^1_{F, \st}(A,B(1))$ by $\Psi$ respectively.
\end{prop}

\subsection{Local $p$-adic height pairings: the $\ell$-adic cases}\label{def_local_height_l-adic}
In the following, we fix a non-trivial continuous homomorphism $l_F:\A_{F}^{\times}/F^{\times}\to \Q_p$, where $\A_{F}^{\times}$ is the id\`{e}le group of $F$. 
For every place $v$ of $F$, let $l_{F,v}$ be the restriction of $l_F$ to $F_v^{\times}$. 
Note that $l_F$ can be extended to $l_F:\A_F^{\times}\hat{\otimes}_{\Z_{p}} E \to E$ whose restriction to $F^{\times}\hat{\otimes}_{\Z_{p}} E$ is zero.
Let $e_1=[\cal{E}_1]\in\Ext^1_{F,\mathrm{st}}(A,V)$ and $e_2=[\cal{E}_2]\in\Ext^1_{F,\mathrm{st}}(V,B(1))$. 
We fix a mixed extension $[\cal{E}] \in \MExt(e_1,e_2)$. 
First, we define local $p$-adic height pairings. 

Let $v \nmid p$. Since $H^1(F_v, E(1)) \cong F_v^{\times}\widehat{\otimes}_{\Z_{p}} E$ by the Kummer map, we can regard $l_{F, v}$ as a continuous homomorphism 
\begin{align}
  l_{F, v}:H^1(F_v, E(1)) \to E
\end{align}
and set 
\begin{align*}
  1 \otimes l_{F, v}:\Hom_{E}(A,B) \otimes_{E} H^1(F_v,E(1))\to \Hom_{E}(A,B).  
\end{align*}
Then, taking a unique extension 
\begin{align}
  [\cal{C}_v]\in \Ext^1_{F_v}(A,B(1))=\Hom_{E}(A,B)\otimes_{E} H^1(F_v, E(1))
\end{align}
appearing in Proposition \ref{prop:l-adic_splitting}, we define the \textit{local $p$-adic height pairing at $v\nmid p$} attached to a mixed extension $[\cal{E}] \in \MExt(e_1,e_2)$ by 
\begin{align}
  h_{v,\cal{E}} \coloneq -(1\otimes l_{F, v})([\cal{C}_v])\in \Hom_{E}(A,B).
\end{align}
The fact that $l_{F, v}$ factors through $\ord_v$ implies the following proposition.

\begin{prop}[{\cite[4.6]{Nek93}}]\label{prop:vanishing_essentially_unramified}
  For $[\cal{E}] \in \MExt_{\st}(e_1,e_2)$, we have $h_{v,\cal{E}} = 0$.
\end{prop}

\subsection{Local $p$-adic height pairings: the $p$-adic cases}\label{def_local_height_p-adic}
Let $L/E$ be a finite extension. 
For each place $v$ above $p$, we fix an $E$-linear splitting $N_v:\ddr(V_v)_L/\ddr^+(V_v)_L\to \ddr(V_v)_L$ of the exact sequence
\begin{align}\label{Hodgefil}
  \xymatrix{
    0 \ar[r] & \ddr^+(V_v)_L \ar[r] & \ddr(V_v)_L \ar[r] & \ddr(V_v)_L/\ddr^+(V_v)_L \ar[r] & 0.  
  }  
\end{align}
Before Proposition \ref{Prop:phi-splittings}, we introduce the functor $\mathscr{D}$.

\begin{dfn}[cf.{\cite[2.6]{Nek93}}]\label{Definition:functor_D}
  Let $\MF_{F_v \otimes_{\Q_{p}} E}$ be the category of filtered modules over $F_v \otimes_{\Q_{p}} E$ and let $\MF_{F_v \otimes_{\Q_{p}} E}(\varphi,N)$ be the category of filtered $(\varphi,N)$-modules over $F_{v, 0} \otimes_{\Q_{p}} E$ (see \cite{Fon79} for the detailed definitions of them). 
  We define the contravariant functor $\mathscr{D}:\MF_{F_v \otimes_{\Q_{p}} E} \to \MF_{F_v \otimes_{\Q_{p}} E}$ by $\mathscr{D}(D) \coloneq \Hom_{F_v \otimes_{\Q_{p}} E}(D, D_{\dR}(E(1)))$ and the contravariant functor $\mathscr{D}_0:\MF_{F_v \otimes_{\Q_{p}} E}(\varphi,N) \to \MF_{F_v \otimes_{\Q_{p}} E}(\varphi,N)$ by $\mathscr{D}_0(D) \coloneq \Hom_{F_{v, 0} \otimes_{\Q_{p}} E}(D, D_{\st}(E(1)))$.
\end{dfn}

The following proposition is the key to the construction of local $p$-adic height pairings and calculation of the Bloch--Kato logarithm map.

\begin{prop}[cf.{\cite[\S\S 3.4-3.5]{Nek93}}]\label{Prop:phi-splittings}
  Let $v \mid p$. 
  For a semistable $p$-adic representation $X$ of $G_{F_v}$ satisfying \textbf{(Hyp 3)} in \S \ref{subsec:set-up} and $[Z]\in\Ext^1_{F_v,\st}(B,X)$, there exists a canonical $(\varphi,N)$-equivariant $(F_{v, 0} \otimes_{\Q_{p}} E)$-linear splitting $s_v: F_{v, 0} \otimes_{\Q_p} B \to \dst(Z)$ of the exact sequence
  \begin{align}\label{seq:V_B(1)}
    \xymatrix{
      0\ar[r] & \dst(X)\ar[r] & \dst(Z)\ar[r] & F_{v, 0} \otimes_{\Q_p} B \ar[r] & 0.
    }
  \end{align}
\end{prop}

\begin{proof}
  The proposition may be well-known but the authors do not know references including the proof. 
  We give a proof for the reader's convenience. 
  Applying the snake lemma to the commutative diagram
  \begin{align}
    \xymatrix{
      0 \ar[r] & \dst(X) \ar[d]^N \ar[r] & \dst(Z) \ar[r] \ar[d]^N & F_{v,0} \otimes_{\Q_p} B \ar[r] \ar[d]^N & 0\\
      0 \ar[r] & \dst(X(-1)) \ar[r] & \dst(Z(-1)) \ar[r] & F_{v,0} \otimes_{\Q_p} B \ar[r] & 0,
    } 
  \end{align}
  we obtain the exact sequence
  \begin{align}
    \xymatrix{
      0\ar[r] & \dcris(X) \ar[r]^{\iota} & \dcris(Z) \ar[r] & F_{v,0} \otimes_{\Q_p} B \ar[r] & \dst(X(-1))/N\dst(X).
    }    
  \end{align}
  This exact sequence yields the two exact sequences
  \begin{align}
    0 \rarrow \dcris(X) \overset{\iota}{\rarrow} & \dcris(Z) \rarrow \Coker\iota \rarrow 0 ,\label{eq:first_seq} \\
    0 \rarrow \Coker\iota \rarrow F_{v,0} \otimes_{\Q_p} &B \rarrow \dst(X(-1))/N\dst(X). \label{eq:second_seq}
  \end{align}
  On the other hand, the assumption that $X$ is semistable and \eqref{eq:cris_v} imply that
  \begin{align}\label{eq:vanishing_quotient}
    (\dst(X(-1))/N\dst(X))^{\varphi=1} &= \left(\Hom_{F_{v,0} \otimes_{\Q_{p}} E}(\dst(X^*(1))^{N=0},\dst(E)) \right)^{\varphi=1}\\
    &=\left(\Hom_{F_{v,0} \otimes_{\Q_{p}} E}(\dcris(X^*(1)),\dst(E)) \right)^{\varphi=1}\\
    &=0.
  \end{align}
  Hence we see that $\varphi-1:\dst(X(-1))/N\dst(X) \to \dst(X(-1))/N\dst(X)$ is an $(E \otimes_{\Q_{p}} \Q_{p})$-linear isomorphism.
  Applying the snake lemma to the commutative diagram 
  \begin{align}
    \xymatrix{
      0 \ar[r] & \Coker\iota \ar[r] \ar[d]^{\varphi-1} & F_{v,0} \otimes_{\Q_p} B \ar[r] \ar[d]^{\varphi-1} & \dst(X(-1))/N\dst(X) \ar[d]^{\varphi-1}_{\cong} \\
      0 \ar[r] & \Coker\iota \ar[r] & F_{v,0} \otimes_{\Q_p} B \ar[r] & \dst(X(-1))/N\dst(X)
    }   
  \end{align}
  obtained by \eqref{eq:second_seq}, we have 
  \begin{align}
    (\Coker\iota)^{\varphi=1}\cong (F_{v,0} \otimes_{\Q_p} B )^{\varphi=1}\cong F_{v,0}^{\varphi=1}\otimes_{\Q_p} B\cong \Q_p\otimes_{\Q_p}B\cong B.
  \end{align}
  Since $\varphi-1:\dcris(X) \to \dcris(X)$ is also an $E$-linear isomorphism by \eqref{eq:cris_v}, it follows that $\dcris(Z)^{\varphi=1} \cong (\Coker\iota)^{\varphi=1}$ from the same argument for \eqref{eq:first_seq} as above.
  By combining these isomorphisms, we obtain an $E$-linear isomorphism $\pi':\dcris(Z)^{\varphi=1} \overset{\sim}{\rightarrow} B$. 
  We define $s_v:F_{v,0} \otimes_{\Q_p} B\to \dst(Z)$ by
  \begin{align}
  s_v(a \otimes b) \coloneq a \pi'^{-1}(b)
  \end{align}
  for $a \otimes b \in F_{v,0} \otimes_{\Q_p} B$. 
\end{proof}

Applying Proposition \ref{Prop:phi-splittings} to $X=V^*(1)_v$ and $Z=\cal{E}_2^*(1)_v$, we obtain the $(\varphi,N)$-equivariant $(F_{v, 0} \otimes_{\Q_{p}} E)$-linear splitting $s_v:F_{v, 0} \otimes_{\Q_p} B^{*} \to \dst(\cal{E}_2^*(1)_v)$ of the exact sequence
\begin{align}\label{seq:splitting_dst}
  \xymatrix{
  0\ar[r] &\dst(V^*(1)_v) \ar[r] & \dst(\cal{E}_2^*(1)_v) \ar[r] & F_{v, 0} \otimes_{\Q_p} B^{*} \ar[r] & 0.
  }
\end{align}
Tensoring \eqref{seq:splitting_dst} with $F_v$, we have a splitting $\bar{s}_v: F_{v} \otimes_{\Q_p} B^{*} \to \ddr(\cal{E}_2^*(1)_v)$ and its corresponding splitting $\bar{s}_v':D_{\dR}(\cal{E}_2^*(1)_v) \to D_{\dR}(V^*(1)_v)$ of the exact sequence
\begin{align}
  \xymatrix{
    0\ar[r] & \ddr(V^*(1)_v) \ar[r] & \ddr(\cal{E}_2^*(1)_v) \ar[r] & F_{v} \otimes_{\Q_p} B^{*} \ar[r] & 0.
  }  
\end{align}
Moreover, taking the functor $\mathscr{D}$, we obtain the splitting $\mathscr{D}(\bar{s}_v'):\ddr(V_v)\to \ddr(\cal{E}_{2,v})$ of the exact sequence
\begin{align}
  \xymatrix{
    0 \ar[r] & F_{v} \otimes_{\Q_p} B \ar[r] & \ddr(\cal{E}_{2,v}) \ar[r] & \ddr(V_v) \ar[r] & 0.
  }
\end{align}
Define an $E$-linear splitting $u_{v}$ by 
\begin{align}
  u_v \coloneq \varpi \circ \mathscr{D}(\bar{s}_v')_L \circ N_v:\ddr(V_v)_L/\ddr^+(V_v)_L\to \ddr(\cal{E}_{2,v})_L/\ddr^+(\cal{E}_{2,v})_L
\end{align}
in the diagram (whose second and third rows are commutative)
\begin{align}
  \xymatrix{
    & &\ddr(\cal{E}_{2,v})_L\ar@{->>}[d]^{\varpi}& \ddr(V_v)_L\ar[l]^{\mathscr{D}(\bar{s}_v')_L}&\\
    0 \ar[r] &  (F_{v} \otimes_{\Q_{p}} B)_{L} \ar[r] \ar[d]^{\exp_{B(1)_v}^{\mathrm{BK}}}_{\cong} & \ddr(\cal{E}_{2,v})_L/\ddr^+(\cal{E}_{2,v})_L \ar[r] \ar[d]^{\exp_{\cal{E}_{2,v}}^{\mathrm{BK}}}_{\cong} & \ddr(V_v)_L/\ddr^+(V_v)_L \ar[r] \ar[u]^{N_v} \ar[d]^{\exp_{V_v}^{\mathrm{BK}}}_{\cong} & 0 \\
    0 \ar[r] & H^1_{f}(F_v,B(1))_L \ar[r] & H^1_{f}(F_v, \cal{E}_{2,v})_L \ar[r] & H^1_{f}(F_v,V_v)_L \ar[r] & 0,
  }  
\end{align}
where $\varpi:\ddr(\cal{E}_{2,v})_L\to \ddr(\cal{E}_{2,v})_L/\ddr^+(\cal{E}_{2,v})_L$ is the canonical surjective map.
Since we have $H^1_{\st}(F_v,V_v) = H^1_{f}(F_v,V_v)$ by the hypothesis \textbf{(Hyp 3)}, the composite of $u_{v}$ and the inclusion $H^{1}_{f}(F_{v}, \cal{E}_{2,v})_{L} \to H^{1}_{\st}(F_{v}, \cal{E}_{2, v})_{L}$ also becomes a splitting of 
\begin{align}
  \xymatrix{
    0 \ar[r] & H^1_{\st}(F_v,B(1))_L \ar[r] & H^1_{\st}(F_v, \cal{E}_{2,v})_L \ar[r] & H^1_{\st}(F_v,V_v)_L \ar[r] & 0. 
  } \label{eq:exact_H1st}
\end{align}
Let $w_v:H^1_{\st}(F_v, \cal{E}_{2,v})_L \to H^1_{\st}(F_v,B(1))_L$ be the splitting of the exact sequence \eqref{eq:exact_H1st} corresponding to $u_{v}$.
Note that
\begin{align}
  [\cal{E}_v] \in \Ext^1_{F_{v}, \st}(A, \cal{E}_{2, v}) \otimes L = \Hom(A,H^1_{\st}(F_{v}, \cal{E}_{2, v})) \otimes L 
\end{align}
and that
\begin{align}
  w_v \circ [\cal{E}_v] & \in \Hom(A, H^1_{\st}(F_v,B(1))) \otimes L \\
  &=\Hom(A,B) \otimes H^1_{\st}(F_v, E(1))_L
  =\Hom(A,B) \otimes (F_{v}^{\times} \widehat{\otimes}_{\Z_{p}} L).
\end{align}

\begin{dfn}[{\cite[7.4]{Nek93}}]
  We define the \textit{local $p$-adic height pairing $h_{v, \cal{E}, N_v} \in \Hom(A,B)_L$ at $v\mid p$} attached to a mixed extension $[\cal{E}] \in \MExt(e_1, e_2)$ and a splitting $N_v$ by
  \begin{align}
    h_{v, \cal{E}, N_v} \coloneq - (1\otimes l_{F, v})(w_v \circ [\cal{E}_v]) \in \Hom(A,B)_L.
  \end{align}
\end{dfn}

\subsection{Global $p$-adic height pairing}\label{sec:global_height}

To define global $p$-adic height pairings, we apply the arguments in \S\ref{def_local_height_l-adic}, \S\ref{def_local_height_p-adic} to $A=H^1_{f}(F,V)$ and $B=H^1_{f}(F,V^*(1))^*$. 
Then, we can identify $\Ext^1_{F,\st}(A,V)\cong \Hom(A,A)$ and $\Ext^1_{F,\st}(V,B(1))\cong \Ext^1_{F,\st}(B^*,V^*(1))\cong \Hom(B^*,B^*)$
by using \eqref{isom:Ext_Hom}. 
Let $e_1=[\cal{E}_1]\in \Ext^1_{F,\st}(A,V)$ be the extension corresponding to the identity map $\mathrm{id}_A$, and let $e_2=[\cal{E}_2]\in \Ext^1_{F,\st}(V,B(1))$ be the extension class such that its dual twist
\begin{align}
    e_2^*(1)=[E_2^*(1)]\in \Ext^1_{F,\st}(B^*,V^*(1))
\end{align}
corresponds to $\mathrm{id}_{B^*}$.
We call $e_1=[\cal{E}_1]$ and $e_2=[\cal{E}_2]$ the \textit{universal extensions} in $\Ext^1_{F,\st}(A,V)$ and $\Ext^1_{F,\st}(V,B(1))$, respectively. 

\begin{dfn}[{\cite[Theorem 4.11]{Nek93}}]\label{dfn:global-height}
  The \textit{(global) $p$-adic height pairing} 
  \begin{align}
    \langle\ ,\ \rangle_{l_F,N}:H^1_{f}(F,V^*(1))\times H^1_{f}(F,V)\to L
  \end{align}
  attached to $l_F$ and $N \coloneq \{N_v\}_{v\mid p}$ is defined by
  \begin{align}\label{def:global_sum_local}
    \langle d, a \rangle_{l_F, N} \coloneq \sum_{v \nmid p}(d \circ h_{v,\cal{E}})(a)+\sum_{v\mid p}(d\circ h_{v,\cal{E},N_v})(a)
  \end{align}
  for $d\in H^1_{f}(F,V^*(1))(=B^*)$ and $a\in H^1_{f}(F,V)$.
\end{dfn}

\begin{rem}
  We can check that the summation in the right hand side of \eqref{def:global_sum_local} converges, and hence \eqref{def:global_sum_local} is well-defined. 
  Indeed, for every mixed extension $[\cal{E}] \in \MExt(e_1, e_2)$ and every mixed extension $[\cal{E}'] \in \MExt_{\st}(e_1, e_2)$, there exists $x = [X] \in \Ext^1_{F}(A, B(1))$ such that $\Psi(x, [\cal{E}']) = [\cal{E}]$ by Proposition \ref{prop:principal_mixed_extension}. 
  By Proposition \ref{prop:vanishing_essentially_unramified}, we see that
  \begin{align}
    \sum_{v} h_{v, \cal{E}} &= \sum_{v} h_{v,\Psi(x,[\cal{E}'])}=\sum_{v} (h_{v, \cal{E}'} - (1 \otimes l_{F, v}) \circ [X_v])\\
    &= \left( \sum_{v} h_{v, \cal{E}'} \right) - (1 \otimes l_F)([X]) = \sum_{v \mid p} h_{v, \cal{E}'}.
  \end{align}
  The second equality follows from \cite[Proposition 4.10 (1)]{Nek93}.
  In particular, the right hand side of \eqref{def:global_sum_local} does not depend on the choice of a mixed extension $[\cal{E}]$ and hence we omit $\cal{E}$ in the notation of the left hand side.
\end{rem}

\section{Proof of main results}
\subsection{Difference arising from splittings}\label{sec:differ_splitting}
We use the same notation as in \S2. 
We calculate the difference of $p$-adic height pairings arising from two splittings to prove Theorem \ref{main_theorem1}. 

For every place $v$ above $p$, let $\mu_v:H^1_{\st}(F_v, \cal{E}_{2,v})_L \to H^1_{\st}(F_v, V_v)_L$ be the canonical surjection. 
Then we have $w_v = 1 - u_v \circ \mu_v$. 

From this section, take two splittings $N=\{ N_v \}_{v\mid p}$ and $N'=\{N_v'\}_{v\mid p}$ as ($F_{v} \otimes_{\Q_{p}} E$)-linear. 
Then, we see that
\begin{align}
  &\langle d, a \rangle_{l_F, N} - \langle d, a \rangle_{l_F, N'} \\ 
  &= \sum_{v \mid p} \{ (d \circ h_{v, \cal{E}, N_v})(a) - (d \circ h_{v, \cal{E}, N_v'})(a) \} \notag \\
  &=  \sum_{v \mid p} \left( (1 \otimes l_{F, v})  \circ (d \otimes 1)\circ (w_v'-w_v)\circ [\cal{E}_{v}] \right)(a)\notag \\
  &= \sum_{v \mid p} \left( (1 \otimes (l_{F, v} \circ \exp_{E(1)}^{\BK}) ) \circ (d\otimes 1)\circ (u_v-u_v') \circ \mu_v \circ [\cal{E}_{v}] \right)(a). \label{eq:difference_height_arising_from_splittings}
\end{align}

\begin{lem}\label{indepent_mixed_extension}
  The map $\mu_v\circ [\cal{E}_{v}]$ does not depend on the choice of $[\cal{E}] \in \MExt(e_1, e_2)$. 
  Moreover, $\mu_v\circ [\cal{E}_{v}](a) = a_v$.
\end{lem}

\begin{proof}
  It suffices to show that $\mu_{v} \circ [\cal{E}_{v}] = [\cal{E}_{1, v}]$ for any $[\cal{E}] \in \MExt(e_{1}, e_{2})$.
  By the definition of the mixed extension, the pushout of a mixed extension $[\cal{E}_{v}] \in \Ext^{1}_{F_v}(A, \cal{E}_{2, v})$ by the canonical surjection $\cal{E}_{2,v} \to V_v$ is equal to $[\cal{E}_{1, v}] \in \Ext_{F_v}^{1}(A, V_{v})$.
  The second assertion follows from the fact that the class $[\cal{E}_{1,v}]$ corresponds to the localization map.
\end{proof}

Let $[W] \in \Ext^1_{F, \st}(V, E(1))$ be the extension such that $[W^{*}(1)]\in \Ext^1_{F, \mathrm{st}}(E, V^*(1)) \cong B^*$ corresponds to $d$. 
The extension $W$ fits into the commutative diagram
\begin{align}
  \xymatrix{
    0 \ar[r] & B(1) \ar[r] \ar[d]_{d\otimes 1} & \cal{E}_2 \ar[r] \ar[d] & V \ar[r] \ar@{=}[d] & 0 \\
    0 \ar[r] & E(1) \ar[r] & W \ar[r] & V \ar[r] & 0.
  }  
\end{align}
Applying Proposition \ref{Prop:phi-splittings} in the same way as $u_v$, we can obtain a ($F_{v} \otimes_{\Q_{p}} E$)-linear splitting $r_v:\ddr(V_v)_L/\ddr^+(V_v)_L\to \ddr(W_v)_L/\ddr^+(W_v)_L$ of the exact sequence
\begin{align}
  \xymatrix{
    0 \ar[r] & (F_{v} \otimes_{\Q_{p}} E)_{L} \ar[r] & \ddr(W_v)_L/\ddr^+(W_v)_L\ar[r] & \ddr(V_v)_L/\ddr^+(V_v)_L\ar[r] &0
  }  
\end{align}
for every $v\mid p$. 
(Note that $r_{v}$ also depends on the choice of a splitting $N_{v}$.)

\begin{lem}\label{lem:difference_via_splittings}
  We have
  \begin{align}
    (d\otimes 1)\circ (u_v-u_v')=r_v-r_v'.
  \end{align}
\end{lem}

\begin{proof}
The  splittings 
\begin{align}
  &u_v:\ddr(V_v)_L/\ddr^+(V_v)_L\to \ddr(\cal{E}_{2,v})_L/\ddr^+(\cal{E}_{2,v})_L,\\ &r_v:\ddr(V_v)_L/\ddr^+(V_v)_L\to \ddr(W_v)_L/\ddr^+(W_v)_L
\end{align}
fit into the commutative diagram
\begin{align}
  \label{diag:before}
  \xymatrix{
    0 \ar[r] & (F_{v} \otimes_{\Q_{p}} B)_{L} \ar[r] \ar[d]^{d\otimes 1} & \ddr(\cal{E}_{2,v})_L/\ddr^+(\cal{E}_{2,v})_L \ar[d] \ar[r] & \ddr(V_v)_L/\ddr^+(V_v)_L \ar[r] \ar@{=}[d] & 0\\
    0 \ar[r] & (F_{v} \otimes_{\Q_{p}} E)_{L} \ar[r] & \ddr(W_v)_L/\ddr^+(W_v)_L \ar[r] & \ddr(V_v)_L/\ddr^+(V_v)_L \ar[r] & 0.
  }  
\end{align}
The lemma follows from \eqref{diag:before} and chasing the diagram.
\end{proof}

Applying the functor $\mathscr{D}$ to the splittings $u_v$ and $r_v$ in \eqref{diag:before}, we also have the splittings 
\begin{align}\label{oshiruko1}
  \mathscr{D}(u_v):\ddr^+(\cal{E}^*_2(1)_v)_L\to \ddr^+(V^*(1)_v)_L, \quad
  \mathscr{D}(r_v):\ddr^+(W^*(1)_v)_L\to \ddr^+(V^*(1)_v)_L
\end{align}
in the commutative diagram
\begin{align}\label{oshiruko2}  
  \xymatrix{
    0 \ar[r] & \ddr^+(V^*(1)_v)_L \ar[r] \ar@{=}[d] & \ddr^+(W^*(1)_v)_L \ar[d] \ar[r] & (F_{v} \otimes_{\Q_{p}} E)_{L} \ar[r] \ar[d] & 0 \\
    0 \ar[r] & \ddr^+(V^*(1)_v)_L \ar[r] & \ddr^+(\cal{E}_2^*(1)_v)_L \ar[r] & (F_{v} \otimes_{\Q_{p}} B^{*})_{L} \ar[r] & 0.
  }  
\end{align}
Before the next lemma, we introduce the set of \textit{rigidified extensions} (cf.\cite{Nek93}). 
We consider an extension $[X^*(1)_v] \in \Ext_{F_{v}, \st}^{1}(E, V^{*}(1)_v)$ and a left splitting $w_v$ of the exact sequence
\begin{align}
  \xymatrix{
    0 \ar[r] & \ddr^+(V^*(1)_v) \ar[r] & \ddr^+(X^*(1)_v) \ar[r] & (F_v \otimes_{\Q_{p}} E) \ar[r] & 0.
  } \label{eq:splitting of extrig}
\end{align}
We call a tuple of these data a \textit{rigidified extension}. 
We denote the set of isomorphism classes $([X^*(1)_v],w_v)$ of rigidified extensions by $\Extrig^1_{F_v,\st}(E, V^*(1)_v)$. 
(Two rigidified extensions are isomorphic if there exists a morphism of extensions which is compatible with these splittings.) 
The set $\Extrig^1_{F_v,\st}(E, V^*(1)_v)$ becomes an $E$-vector space as follows. 
We define the sum of two rigidified extensions $([X_1^*(1)_v],w_{1,v}), ([X_2^*(1)_v],w_{2,v})\in \Extrig^1_{F_v,\st}(E, V^*(1)_v)$ by
\begin{align}
    ([X_1^*(1)_v],w_{1,v})+([X_2^*(1)_v],w_{2,v}):=([X_1^*(1)_v]\wedge [X_2^*(1)_v],w_{1,v}\wedge w_{2,v}),
\end{align}
where
\begin{align}
(w_{1,v}\wedge w_{2,v})([(x_1,x_2)]):=w_{1,v}(x_1)+ w_{2,v}(x_2)
\end{align}
for $[(x_1,x_2)]\in \ddr^+(X_1^*(1)_v\wedge X_2^*(1)_v)\cong \ddr^+(X_1^*(1)_v)\wedge \ddr^+(X_2^*(1)_v)$.
The scalar multiplication is defined by, for $a\in E$,
\begin{align}
  a \cdot ([X^*(1)_v],w_v) \coloneq \begin{cases}
      (a\cdot[X^*(1)_v],w_v)&(a\neq0),\\
      ([E\oplus V^*(1)_v],w_{v,0})&(a=0),
  \end{cases}
\end{align}
where 
\begin{align}
w_{v,0}:(F_v\otimes_{\Q_p}E)\oplus \ddr^+(V^*(1)_v)\to \ddr^+(V^*(1)_v)
\end{align}
is the second projection.
Let $([X^*(1)_v],w_v) \in \Extrig^1_{F_v,\st}(E, V^*(1)_v)$. 
From Proposition \ref{Prop:phi-splittings}, we obtain a canonical $(\varphi, N)$-equivariant splitting $s_v:F_{v,0} \otimes_{\Q_{p}} E \to \dst(X^*(1)_v)$ and extend it to $s_v:F_{v} \otimes_{\Q_{p}} E \to \ddr(X^*(1)_v)$ by tensoring with $F_v$. 
We define the well-defined $E$-linear map $\Phi:\Extrig^1_{F_v, \st}(E, V^*(1)_v) \to \ddr(V^*(1)_v)$ by
\begin{align}
  \Phi([X^*(1)_v], w_v) \coloneq w_v'(1) - s_v(1) \in \ddr(V^*(1)_v),
\end{align}
where $w_v':(F_v \otimes_{\Q_{p}} E) \to \ddr^+(X^{*}(1)_v)$ is the corresponding splitting to $w_v$.

The following lemma, which is essential for the computation of the Bloch–Kato logarithm map, is the semistable analogue of \cite[Lemma 2.7]{Nek93}, but can be proved by the same argument as in the crystalline case. 

\begin{lem}[cf.{\cite[Lemma 2.7]{Nek93}}]\label{prop:pairing_bklog}
  The map $\Phi$ is an isomorphism, and the following diagram is commutative:
  \begin{align}\label{diag:BKlog_dR}
    \xymatrix{
      \Extrig^1_{F_v,\st}(E, V^*(1)_v) \ar^{\pi}[r] \ar[d]_{\Phi} & \Ext^1_{F_v,\st}(E, V^*(1)_v) \ar[d]^{(\exp_{V^*(1)_v}^{\mathrm{BK}})^{-1}} \\
      \ddr(V^*(1)_v) \ar@{->>}[r] & \ddr(V^*(1)_v)/\ddr^+(V^*(1)_v),
    }
  \end{align}
  where $\pi$ is the projection to the first component. 
\end{lem}
We return to our setting. 
We apply the above argument to $([W^*(1)_v], w_{v}) \in \Extrig^1_{F_v,\st}(E, V^*(1)_v)$, where $w_{v}$ is any left splitting of \eqref{eq:splitting of extrig}, and put $x_v \coloneq w_{v}'(1)$ and $x_v^{\varphi} \coloneq s_v(1)$. (Hence, 
\begin{align}
    \Phi: \Extrig^1_{F_v,\st}(E, V^*(1)_v) \to \ddr(V^*(1)_v)
\end{align}
sends $([W^*(1)_v], w_{v})$ to $x_v-x_v^{\varphi}$.)

\begin{lem}\label{lemma:deRhampairing_bklog}
Let $v \mid p$ and $[\ ,\ ]_{\dR}: \ddr(V^*(1)_v) \times \ddr(V_v) \to D_{\dR}(E(1)) = F_{v} \otimes_{\Q_{p}} E$ be the de Rham pairing induced by the natural pairing $V^*(1) \times V \to E(1)$. For any $\omega \in \ddr^+(V_v)_L$, we have
  \begin{align}
    [x_v-x_v^{\varphi},\omega]_{\dR}=\bklog_{V^*(1)_v}(d_v)(\omega),
  \end{align}
  where $\bklog_{V^*(1)_v}$ is defined in \eqref{definition:Bloch-Kato_log}. 
\end{lem}

\begin{proof}
We identify
\begin{align}
 \Ext^1_{F_v, \st}(E, V^*(1)_v) \cong H^1_{f}(F_v, V^*(1)_v), \quad \ddr(V^*(1)_v)/\ddr^+(V^*(1)_v)\cong \Hom(\ddr^+(V_v), D_{\dR}(E(1)))
\end{align}
by \eqref{isom:Ext_Hom} and \eqref{definition:Bloch-Kato_log} respectively.
The composition of $\Phi$ and the bottom horizontal map in \eqref{diag:BKlog_dR} sends $([W^*(1)_v], w_{v})$ to the map $[x_v-x_v^{\varphi}, \ast]_{\dR} \in \Hom(\ddr^+(V_v), D_{\dR}(E(1)))$.
Since $\pi(([W^*(1)_v], w_{v}))=d_v \in H^1_{f}(F_v,V^*(1)_v)$, the lemma follows from the definition of $\bklog_{V^*(1)_v}$.
\end{proof}

\begin{prop}
  Assume that $D_{\dR}^{+}(V_{v})$ is a free $(F_{v} \otimes_{\Q_{p}} E)$-module of rank $g$.
  Let $\{\omega_{v,1},\ldots,\omega_{v,g}\}$ be an $(F_v \otimes_{\Q_{p}} E)$-basis of $\ddr^+(V_v)$. 
  Choose $\eta_{v,1},\ldots,\eta_{v,g} \in D_{\dR}(V^*(1)_v)$ whose images in $D_{\dR}(V^{*}(1)_{v})/D_{\dR}^{+}(V^{*}(1)_{v})$ form the dual basis under the de Rham pairing, that is $[\eta_{v,i},\omega_{v,j}]_{\dR}=\delta_{ij}$. 
  Here, $\delta_{ij}$ is the Kronecker delta. 
  Then we have
  \begin{align}\label{eq:difference_Log}
    \mathscr{D}(r_v)(x_v)-\mathscr{D}(r_v')(x_v)=\sum_{j=1}^g\bklog_{V^*(1)_v}(d_v)(\omega_{v,j})(\mathscr{D}(N_v)(\eta_{v,j})-\mathscr{D}(N'_v)(\eta_{v,j})).
  \end{align}
\end{prop}

\begin{proof}
  Lemma \ref{lemma:deRhampairing_bklog} implies that
  \begin{align}
    \left[x_v-x_v^{\varphi}-\sum_{j=1}^g(\bklog_{V^*(1)_v}(d_v)(\omega_{v,j}))\eta_{v,j},\omega_{v,i}\right]_{\dR}&=[x_v-x_v^{\varphi},\omega_{v,i}]_{\dR}-\sum_{j=1}^g\bklog_{V^*(1)_v}(d_v)(\omega_{v,j})[\eta_{v,j},\omega_{v,i}]_{\dR}\\[-15pt]
    &=\bklog_{V^*(1)_v}(d_v)(\omega_{v,i})-\bklog_{V^*(1)_v}(d_v)(\omega_{v,i})=0
  \end{align}
  for $1\leq i\leq g$. 
  Hence, we see that 
  \begin{align}
    x_v-x_v^{\varphi}-\sum_{j=1}^g(\bklog_{V^*(1)_v}(d_v)(\omega_{v,j}))\eta_{v,j}\in \ddr^+(V^*(1)_v)_{L}.
  \end{align}
  In other words, if $\{\omega_{v,i}'\}_i$ is a fixed $(F_v \otimes_{\Q_{p}} E)$-basis of $\ddr^+(V^*(1)_v)$, one can write  
  \begin{align}
    x_v-x_v^{\varphi}-\sum_{j=1}^g(\bklog_{V^{*}(1)_{v}}(d_v)(\omega_{v,j}))\eta_{v,j}=\sum_{i}c_i\omega_{v,i}',
  \end{align}
  for some $c_i \in (F_v \otimes_{\Q_{p}} E)_{L}$. 
  We consider the section $\sigma_v:\ddr(W^*(1)_v)_L\to\ddr(V^*(1)_v)_L$ induced by $s_v$ in the exact sequence
  \begin{align}\label{exsqW*}
    \xymatrix{
      0\ar[r] & \ddr(V^*(1)_v)_L\ar[r] & \ddr(W^*(1)_v)_L\ar[r] & (F_v \otimes_{\Q_{p}} E)_{L} \ar[r] & 0.
    }  
  \end{align}
  By the definition of $r_v$ and the construction of $\sigma_v$, the following diagram is commutative:
  \begin{align}\label{diag:r_v_N_v}
    \xymatrix{
      \ddr^+(V^*(1)_v)_L&\ddr^+(W^*(1)_v)_L\ar[l]^{\mathscr{D}(r_v)}\ar@{^{(}->}[d]\\
      \ddr(V^*(1)_v)_L\ar[u]^{\mathscr{D}(N_v)}&\ddr(W^*(1)_v)_L\ar[l]^{\sigma_v}.
    }    
  \end{align}
  Therefore we have
  \begin{align}
    \mathscr{D}(r_v)(x_v)-\mathscr{D}(r_v')(x_v)&=\left(\mathscr{D}(N_v)(\sigma_v(x_v^{\varphi}))+\sum_{j=1}^g\bklog_{V^*(1)_v}(d_v)(\omega_{v,j})\mathscr{D}(N_v)(\eta_{v,j})+\sum_ic_i\omega_{v,i}'\right) \\
    &\ \ \ \ \ -\left(\mathscr{D}(N_v')(\sigma_v(x_v^{\varphi}))+\sum_{j=1}^g\bklog_{V^*(1)_v}(d_v)(\omega_{v,j})\mathscr{D}(N_v')(\eta_{v,j})+\sum_ic_i\omega_{v,i}'\right)\\
    &=\sum_{j=1}^g\bklog_{V^*(1)_v}(d_v)(\omega_{v,j})(\mathscr{D}(N_v)(\eta_{v,j})-\mathscr{D}(N_v')(\eta_{v,j})).
  \end{align}
  Here, we used the facts that $\mathscr{D}(r_v),\mathscr{D}(r_v'):\ddr^+(W^*(1)_v)_L\to\ddr^+(V^*(1)_v)_L$ are splittings of the exact sequence \eqref{exsqW*}
  in the first equality, and that 
  \begin{align}
    \sigma_v(x_v^{\varphi})=\sigma_v(s_v(1))=0
  \end{align}
  in the second equality.
\end{proof}

\subsection{The main results}\label{settings_main_theorem}
In the following, we assume that $V$ is a $2$-dimensional vector space over $E$ and that $V$ has a non-degenerate skew-symmetric and $G_F$-equivariant pairing $[\ ,\ ]:V\times V\to E(1)$. 
The pairing $[\ ,\ ]$ induces the isomorphism $j:V\to V^*(1)$ which sends $v\in V$ to the linear map $[v, *]$. 
Let $[\ ,\ ]_{\dR}:\ddr(V_v)_{L} \times \ddr(V_v)_{L} \to D_{\dR}(E(1))_{L} = (F_v \otimes_{\Q_{p}} E)_{L}$ be the de Rham pairing induced by $[\ ,\ ]$. 
Note that $[\ ,\ ]_{\dR}$ is also skew-symmetric and that $D_{\dR}^{+}(V_v)$ is a free $(F_{v} \otimes_{\Q_{p}} E)$-module of rank $1$. 
For such a representation $V$, we can naturally view the global $p$-adic height pairing defined in Definition \ref{dfn:global-height} as
\begin{align}
   \langle\ ,\ \rangle_{l_F,N}:H^1_{f}(F,V)\times H^1_{f}(F,V)\to L 
\end{align}
via the identification $j_* : H^1_{f}(F,V) \overset{\sim}{\to} H^1_{f}(F,V^*(1))$.

Fix an $(F_v \otimes_{\Q_{p}} E)$-basis $\omega_{v}$ of $D_{\dR}^{+}(V_{v})$ for any $v \mid p$.
We also assume that
{
  \setlength{\leftmargini}{48pt}
  \begin{enumerate}
    \setcounter{enumi}{3}
    \renewcommand{\labelenumi}{\textbf{(Hyp \arabic{enumi})}}
    \item $\{ \phi \omega_{v}, \omega_{v} \}$ forms an $(F_v \otimes_{\Q_{p}} E)$-basis of $D_{\dR}(V_{v})$ for any $v\mid p$.
\end{enumerate}
}
Here, $\phi \coloneq \varphi^{[F_{v, 0}: \Q_{p}]}$.
Note that $[\phi\omega_v, \omega_v]_{\dR}$ is a unit of the ring $F_{v} \otimes_{\Q_{p}} E$ by this condition.
Since $V_v$ is a two-dimensional semistable representation, we can write 
\begin{align}
  \det_{F_{v, 0} \otimes_{\Q_{p}} E} \left( X - \phi ~|~ D_{\st}(V_{v}) \right) = X^{2} - c_{1, v}X -c_{2, v} \label{eq:charpoly}
\end{align}
for some $c_{1, v}, c_{2, v} \in F_{v, 0} \otimes_{\Q_{p}} E$.
By the hypothesis \textbf{(Hyp 4)}, $c_{1, v}$ and $c_{2, v}$ satisfy $\phi^{2}\omega_{v} = c_{1, v}\phi\omega_{v} + c_{2, v} \omega_{v}$.
However, the following lemma shows that $c_{1, v}, c_{2, v} \in E$.
We are deeply grateful to Kentaro Nakamura for valuable discussions concerning the proof of the following lemma.

\begin{lem} \label{lem:coeff of char poly}
  Let $c_{1, v}$, $c_{2, v}$ be in the above characteristic polynomial \eqref{eq:charpoly}.
  Then, there exist $c_{1, v}', c_{2, v}' \in E$ such that the canonical map $E \to F_{v, 0} \otimes_{\Q_{p}} E$ sends $c_{1, v}'$ and $c_{2, v}'$ to $c_{1, v}$ and $c_{2, v}$, respectively.
\end{lem}
\begin{proof}
  Take a finite extension $E'/E$ sufficiently large so that $F_{v, 0} \otimes_{\Q_{p}} E' \cong \prod_{\tau} E'$, where the product on the right hand side runs over the embeddings $\tau: F_{v, 0} \hookrightarrow E'$.
  If we let $(c_{1, \tau})_{\tau}$ and $(c_{2, \tau})_{\tau}$ be the images of $c_{1, v}$ and $c_{2, v}$ under the map
  \begin{align}
    F_{v,0}\otimes_{\Q_p}E\to F_{v,0}\otimes_{\Q_p}E'\cong \prod_{\tau}E',
  \end{align}
  respectively, then we first show that $c_{1, \tau} = c_{1, \tau'}$ and $c_{2, \tau} = c_{2, \tau'}$ for any $\tau$, $\tau'$.
  Using the above isomorphism, there exist two-dimensional $E'$-vector spaces $D_{i}$ for each $i = 1, 2, \ldots, [F_{v, 0}:\Q_{p}]$ such that $D_{\st}(V_{v})\otimes_EE' \cong \prod_{i=1}^{[F_{v, 0}:\Q_{p}]} D_{i}$ and that $\varphi$ sends $D_{i}$ to $D_{i+1}$.
  Let $\phi_{i} \coloneq \mapres{\phi}{D_{i}}$ be an $E'$-linear map on $D_{i}$ for each $i$.
  For any $i, j$, since $\varphi^{j - i}:D_{i} \to D_{j}$ is an $E'$-linear isomorphism satisfying $\varphi^{j-i} \circ \phi_{i} = \phi_{j} \circ \varphi^{j - i}$, we have $\det_{E'}\left( X - \phi_{i} ~|~ D_{i} \right) = \det_{E'}\left( X - \phi_{j} ~|~ D_{j} \right)$.
  Thus, we see that $c_{1, i} = c_{1, j}$ and $c_{2, i} = c_{2, j}$ for any $i, j$. 
  
  Hence the images of $c_{1, v}$ and $c_{2, v}$ in $F_{v,0} \otimes_{\Q_p} E'$ lie in $1 \otimes_{\Q_{p}} E'$. 
  Since $E'$ is faithfully flat over $E$, the natural map
  \begin{align}
    (F_{v,0}\otimes_{\Q_p}E)/(1\otimes_{\Q_{p}} E)\to (F_{v,0}\otimes_{\Q_p}E')/(1\otimes_{\Q_{p}} E')
  \end{align}
  is injective. Therefore $c_{1, v}, c_{2, v} \in 1 \otimes E$.
\end{proof}

By this lemma, we identify $c_{1, v}$, $c_{2, v}$ with elements of $E$.
Let $\alpha_v, \beta_v \in \overline{E}$ be the roots of the polynomial $X^2-c_{1, v} X-c_{2, v}$, and let $L/E$ be a finite extension containing $\alpha_v$ and $\beta_v$ for all $v\mid p$. 
We define the $(F_{v} \otimes_{\Q_{p}} L)$-submodules $N_{\alpha_v}$ and $N_{\beta_v}$ of $\ddr(V_v)_L$ of rank 1 by
\begin{align}\label{def:splitting_Hodge_filtration}
  N_{\alpha_v} \coloneq (F_v\otimes_{\Q_p}L)(\beta_v\omega_v-\phi\omega_v),\ N_{\beta_v} \coloneq (F_v\otimes_{\Q_p}L)(\alpha_v\omega_v-\phi\omega_v).
\end{align}
Since $\{\phi\omega_v, \omega_{v}\}$ is an $(F_{v} \otimes_{\Q_{p}} L)$-basis of $\ddr(V_v)_L$, we obtain the following lemma.
\begin{lem}\label{lem:proof_splitting}
  The module $\ddr(V_v)_L$ admits two decompositions:
  \begin{align}
    \ddr(V_v)_L=\ddr^+(V_v)_L\oplus N_{\alpha_v}=\ddr^+(V_v)_L\oplus N_{\beta_v}.
  \end{align} 
\end{lem}
We identify the modules $N_{\alpha_v}$ and $N_{\beta_v}$ with the $(F_v \otimes_{\Q_{p}} L)$-linear maps 
\begin{align}
    N_{\alpha_v}, N_{\beta_v}:\ddr(V_v)_L/\ddr^+(V_v)_L \to \ddr(V_v)_L
\end{align}
defined by $N_{\alpha_v}(\phi\omega_v) \coloneq \phi\omega_v - \beta_v\omega_v$ and $N_{\beta_v}(\phi\omega_v) \coloneq \phi\omega_v - \alpha_v\omega_v$ respectively. 
By Lemma \ref{lem:proof_splitting}, $N_{\alpha_v}$ and $N_{\beta_v}$ are splittings of the Hodge filtration \eqref{Hodgefil}.
The following theorem is our main result.

\begin{thm}\label{thm:MainTheorem} 
  Let $V$ be a two-dimensional self-dual symplectic $p$-adic Galois representation of $G_F$ satisfying \textbf{(Hyp 1)}-\textbf{(Hyp 4)}. 
  We put $N_{\alpha} \coloneq \{N_{\alpha_v}\}_{v\mid p}$ and $N_{\beta} \coloneq \{N_{\beta_v}\}_{v\mid p}$. 
  Then we have
  \begin{align}\label{eq:differ_height1}
    &\langle d,a\rangle_{l_F,N_{\alpha}}-\langle d,a\rangle_{l_F,N_{\beta}} \\
    &=\sum_{v\mid p}(\alpha_{v} - \beta_{v})(l_{F, v} \circ \exp_{E(1)}^{\rm BK})([\phi\omega_v, \omega_v]_{\dR}^{-1}\cdot\bklog_{V_v}(d_v)(\omega_v)\cdot\bklog_{V_v}(a_v)(\omega_v))
  \end{align}
  for $a,d\in H^1_{f}(F,V)$. 
  Here, $\cdot$ is the multiplication in the ring $F_{v} \otimes_{\Q_{p}} L$.
  In particular, we have
  \begin{align}\label{eq:differ_height2}
    &\langle a, a \rangle_{l_F,N_{\alpha}} - \langle a, a \rangle_{l_F, N_{\beta}} \\
    &= \sum_{v\mid p}(\alpha_{v} - \beta_{v})(l_{F, v} \circ \exp_{E(1)}^{\rm BK})([\phi\omega_v, \omega_v]_{\dR}^{-1}\cdot\bklog_{V_v}(a_v)(\omega_v)^2).
  \end{align}
\end{thm}

\begin{proof}
  By Lemma \ref{lem:difference_via_splittings} and \eqref{eq:difference_height_arising_from_splittings}, it holds that
  \begin{align}
    \langle d, a \rangle_{l_F, N_{\alpha}} - \langle d, a \rangle_{l_F, N_{\beta}} = \sum_{v \mid p}(l_{F, v} \circ \exp_{E(1)}^{\BK} \circ (r_v - r_v') \circ \mu_v \circ [\cal{E}_v])(a).
  \end{align}
  Fix a place $v$ of $F$ above $p$. 
  By Lemma \ref{indepent_mixed_extension},
  \begin{align}
    (r_v-r_v')\circ \mu_v\circ [\cal{E}_v](a)=(r_v-r_v')((\exp_{V_v}^{\mathrm{BK}})^{-1}(a_v))=[\mathscr{D}(r_v)(x_v)-\mathscr{D}(r_v')(x_v), (\exp_{V_v}^{\mathrm{BK}})^{-1}(a_v)]_{\dR}.
  \end{align}
  Here, the second equality follows from the commutative diagram
  \begin{align}
    \xymatrix{ 
      \ddr(V_v)_L/\ddr^+(V_v)_L\ar[d]_{\cong}\ar[rr]^{\qquad r_v-r'_v}&&(F_{v} \otimes_{\Q_{p}} E)_{L}\ar[d]^{\cong}\\
      \Hom(\ddr^+(V_v)_L, D_{\dR}(E(1))_L)\ar[rr]^{\mathscr{D}(r_v-r'_v)^*}&&\Hom(\ddr^+(E)_L, D_{\dR}(E(1))_L).
    }    
  \end{align}
  We calculate $l_{F, v}\circ \exp_{E(1)}^{\BK}([\mathscr{D}(r_v)(x_v) - \mathscr{D}(r_v')(x_v), (\exp_{V_{v}}^{\BK})^{-1}(a_v)]_{\dR})$.  
  Since the pairing $[\ ,\ ]$ is skew-symmetric, we have
  \begin{align}\label{eq:alt_vanish}
    [\phi\omega_v - \beta_v\omega_v, \phi\omega_v - \beta_v\omega_v]_{\dR}=0.
  \end{align}
  From the fact that the functor $\mathscr{D}$ is defined by the de Rham pairing, \eqref{eq:alt_vanish} implies that $\mathscr{D}(N_{\alpha_v})(\phi\omega_v - \beta_v\omega_v) = 0$. 
  Thus we obtain $\mathscr{D}(N_{\alpha_v})(\phi \omega_v)=\mathscr{D}(N_{\alpha_v})(\beta_v\omega_v)=\beta_v\omega_v$ since $\omega_v \in \ddr^+(V_v)_L$.
  Put $\eta_v \coloneq [\phi\omega_v, \omega_v]_{\dR}^{-1}\phi\omega_v\in\ddr(V_v)_L$. 
  Then $\eta_v$ satisfies that $[\eta_v, \omega_{v}]_{\dR}=1$ and
  we have
  \begin{align}
    \mathscr{D}(N_{\alpha_v})(\eta_v) = [\phi\omega_v, \omega_v]_{\dR}^{-1}\mathscr{D}(N_{\alpha_v})(\phi\omega_v) = [\phi\omega_v, \omega_v]_{\dR}^{-1}\beta_{v}\omega_v.
  \end{align}
  Similarly, we also obtain
  \begin{align}
    \mathscr{D}(N_{\beta_v})(\eta_v) = [\phi\omega_v, \omega_v]_{\dR}^{-1}\alpha_{v} \omega_v.
  \end{align}
  By \eqref{eq:difference_Log}, we have
  \begin{align}
    \mathscr{D}(r_v)(x_v)-\mathscr{D}(r_v')(x_v)&=-\bklog_{V_v}(d_v)(\omega_v)(\mathscr{D}(N_{\alpha_v})(\eta_v)-\mathscr{D}(N_{\beta_v})(\eta_v)) \\
    &= (\alpha_{v} - \beta_{v}) [\phi\omega_v, \omega_v]_{\dR}^{-1} \bklog_{V_v}(d_v)(\omega_v)\omega_v.
  \end{align}
  Therefore 
  \begin{align}
    &(l_{F, v} \circ \exp_{E(1)}^{\BK})([\mathscr{D}(r_v)(x_v)-\mathscr{D}(r_v')(x_v), (\exp_{V_{v}}^{\BK})^{-1}(a_v)]_{\dR}) \\
    &=(\alpha_{v} - \beta_{v}) \left( (l_{F, v} \circ \exp_{E(1)}^{\BK})([\phi\omega_v, \omega_v]_{\dR}^{-1} \bklog_{V_v}(d_v)(\omega_v)[\omega_v,(\exp_{V_v}^{\mathrm{BK}})^{-1}(a_v)]_{\dR}) \right) \\
    &=(\alpha_{v} - \beta_{v}) \left( (l_{F, v} \circ \exp_{E(1)}^{\BK}) ([\phi\omega_v, \omega_v]_{\dR}^{-1} \bklog_{V_v}(d_v)(\omega_v)\bklog_{V_v}(a_v)(\omega_v)) \right)
  \end{align}
  by the definition of $\bklog_{V_v}$ given in \eqref{definition:Bloch-Kato_log}.
\end{proof}

By applying Theorem \ref{thm:MainTheorem} to the cyclotomic logarithm $l_F^c$, we can prove non-triviality of $p$-adic height pairings in the next corollary. 
First, we recall the definition of $l_F^c$. 
For every place $v$ of $F$, the $v$-th component $l_{F,v}^c$ of $l_F^c$ is defined by
\begin{align}
  l_{F,v}^c(x) \coloneq \begin{cases}
    \ord_{v}(x)\log_p(q_v)&(v\nmid p\infty), \\
    -\log_p\norm_{F_v/\Q_p}(x)&(v\mid p),\\
    0&(v\text{ ; archimedean}),
  \end{cases}
\end{align}
where $q_{v}$ is the cardinality of the residue field of $F_{v}$, $\log_p$ is the $p$-adic logarithm on $\Q_p^{\times}$ satisfying $\log_{p}p=0$ and $\norm_{F_v/\Q_p}:F_v^{\times} \to \Q_p^{\times}$ is the norm map.

Before stating the following corollary, we recall the definition of the Witt index of a quadratic space.
Let $W$ be a finite-dimensional vector space over a field $K$ of characteristic 0, and $q:W \to K$ be a non-degenerate quadratic form.
A subspace $W'$ of $W$ is said to be totally isotropic if $q(x) = 0$ for every $x \in W'$.
We define the Witt index $i_{q}$ of $q$ to be the maximum of $\dim_{K} W'$, where $W'$ ranges over all totally isotropic subspaces of $W$.
Note that $i_{q} \leq (\dim_{K} W)/2$.
(See \cite{Voi}, for example.)

For a prime $v \mid p$ of $F$, let $\tau: F_{v} \otimes_{\Q_{p}} E \to E$ be the $E$-linear extension of the trace map $\Tr_{F_{v}/\Q_{p}}$, and define a quadratic form $Q$ on $F_{v} \otimes_{\Q_{p}} E$, regarded as an $E$-vector space, by $Q(x) = \tau([\phi\omega_{v}, \omega_{v}]_{\dR}^{-1} x^{2})$.
Since the trace pairing is non-degenerate and $[\phi\omega_v, \omega_v]_{\dR}$ is a unit of $F_{v} \otimes_{\Q_{p}} E$, the quadratic form $Q$ is non-degenerate.
In particular, we see that $i_{Q} \leq [F_{v}:\Q_{p}]/2$.

\begin{cor}\label{cor:non-triviality}
  In the situation of Theorem \ref{thm:MainTheorem}, we assume that there exists a unique prime $v$ of $F$ above $p$, that $\alpha_v\neq \beta_v$, and that $\dim_{E} (\Im (\loc_{v})) > i_{Q}$, where $\loc_{v}:H^{1}_{f}(F, V) \to H^{1}_{f}(F_{v}, V_{v})$ is the localization map at $v$.
  Then either $\langle\ ,\ \rangle_{l_F^c,N_{\alpha_v}}$ or $\langle\ ,\ \rangle_{l_F^c,N_{\beta_v}}$ is non-trivial. 
  In particular, if $X^2-c_{1, v}X-c_{2, v}$ is irreducible over $E$, both $\langle\ ,\ \rangle_{l_F^c,N_{\alpha_v}}$ and $\langle\ ,\ \rangle_{l_F^c,N_{\beta_v}}$ are non-trivial.
\end{cor}

\begin{rem}
  Put $d \coloneq [F_{v} : \Q_p]$.
  Since $Q$ is a $d$-dimensional quadratic form over $E$, its Witt index $i_{Q}$ is determined by the dimension, the discriminant, and the Hasse--Witt invariant of $Q$, all of which can be read off from the Gram matrix of $Q$ with respect to an $E$-basis of $F_{v} \otimes_{\Q_{p}} E$.
    
  If $d = 1$, that is $F_{v} = \Q_{p}$, then $Q$ is one-dimensional, hence anisotropic, and $i_{Q} = 0$.
  Thus the hypothesis $\dim_{E}(\Im(\loc_{v})) > i_{Q}$ amounts to the non-vanishing of $\loc_{v}$, which yields Corollary \ref{cor:main_theorem}.
    
  If $d = 2$, write $F_{v} = \Q_{p}(\sqrt{\delta})$ and let $N_{F_{v} \otimes_{\Q_p} E/E}$ be the norm map of $F_{v} \otimes_{\Q_p} E$ over $E$.
  Computing the Gram matrix with respect to the basis $\{ 1,  \sqrt{\delta} \}$, we find $\det Q \equiv \delta N_{F_{v} \otimes_{\Q_p} E/E}([\varphi\omega_{v}, \omega_{v}]_{\mathrm{dR}}) \bmod (E^{\times})^2$.
  Since a binary form is isotropic if and only if the negative of its discriminant is a square, we have $i_{Q} = 1$ if $-\delta N_{F_{v} \otimes_{\Q_p} E/E}([\varphi\omega_{v}, \omega_{v}]_{\dR}) \in (E^{\times})^2$, and $i_{Q} = 0$ otherwise.
\end{rem}

\begin{proof}[Proof of Corollary \ref{cor:non-triviality}]
  Note that $l_{F, v}^{c} \circ \exp_{E(1)}^{\BK} = - \tau$ by the definition of $l_{F}^{c}$.
  Let $\lambda:H^{1}_{f}(F, V) \to F_{v} \otimes_{\Q_{p}} E$ be the map defined by $\lambda(a) = \log_{V_{v}}^{\BK}(a_{v})(\omega_{v})$, and $W \coloneq \Im \lambda$.
  Then we see that $\dim_{E} W = \dim_{E} (\Im(\loc_{v}))$ since $\log_{V_{v}}^{\BK}$ and the map given by pairing with $\omega_v$ are both isomorphisms.
  Since $\dim_{E} W > i_{Q}$, the subspace $W$ is not totally isotropic, so there exists $a \in H^{1}_{f}(F, V)$ such that $Q(\lambda(a)) \neq 0$.  
  By the assumption $\alpha_{v} \neq \beta_{v}$ and \eqref{eq:differ_height2}, we find that 
  \begin{align}
    \langle a,a\rangle_{l_F^c,N_{\alpha_v}}-\langle a,a\rangle_{l_F^c,N_{\beta_v}}
    &= (\alpha_{v} - \beta_{v}) l_{F, v}^{c} \left( \exp_{E(1)}^{\BK} \left( [\phi\omega_{v}, \omega_{v}]_{\dR}^{-1} \lambda(a)^{2} \right) \right) \\
    &= (\beta_{v} - \alpha_{v}) Q(\lambda(a)) \neq 0.
  \end{align}
  It follows that either $\langle\ ,\ \rangle_{l_F^c,N_{\alpha_v}}$ or $\langle\ ,\ \rangle_{l_F^c,N_{\beta_v}}$ is non-trivial. 
  When $X^2-c_{1, v}X-c_{2, v}$ is irreducible over $E$, $\alpha_v$ and $\beta_v$ are conjugate; hence $N_{\alpha_v}$ and $N_{\beta_v}$ are conjugate as well. 
  Consequently, $\langle\ , \ \rangle_{l_F^c, N_{\alpha_v}}$ and $\langle\ , \ \rangle_{l_F^c,N_{\beta_v}}$ are also conjugate by their constructions, which completes the proof of the latter part of the corollary.
\end{proof}

\begin{ex}
  Put $F = \Q(\sqrt{7})$ and $p = 5$.  
  Then, $p$ is inert in $F/\Q$, and therefore $[F_{v}:\Q_{p}] = 2$ where $v$ is the unique prime of $F$ above $p$.
  Let $A$ be the elliptic curve defined by $y^{2} = x^{3} -x + 1$, and consider the $p$-adic representation $V=V_{p}A$ of $G_{F}$.
  Then $A$ has good reduction at $v$, so that $V$ is a crystalline representation at $v$.
  Since $A$ is non-CM, $V$ satisfies \textbf{(Hyp 4)} by Serre--Tate (cf. \cite{Ser89}).
  The Hecke polynomial of $A$ at $v$ has two roots $-3\pm 4\sqrt{-1}$.
  Using the Nagell--Lutz theorem, we can check that the Mordell--Weil group $A(F)$ has two points $P_{1} = (0,1)$ and $P_{2} = (2, \sqrt{7})$ of infinite order. 
  Let $\sigma \in \Gal(F/\Q)$ be the non-trivial element, and $\omega_{A} = dx/2y$ be the invariant differential of $A$.  
  Since $\sigma(P_1)=P_1$ and $\sigma(P_2)=-P_2$, $P_1$ and $P_2$ are independent, and $\log_{\omega_A}(P_1)\in\Q_5$ and $\log_{\omega_A}(P_2)\in \sqrt{7}\Q_5$. 
  Via the local Kummer map, the images of $P_1$ and $P_2$ therefore form a $\Q_5$-basis of $H^1_f(F_v,V_v)$.
  Therefore, the localization map $H^{1}_{f}(F, V) \to H^{1}_{f}(F_{v}, V_{v})$ is surjective, in particular, $\dim_{\Q_{5}}(\Im(\loc_{v})) > i_{Q}$.
\end{ex}

\section{Applications}\label{sec:Applications}
\noindent 
In this section, we generalize the results of \cite[Appendix A]{BKO24} to higher weight modular forms. 
Let $f = \sum_{n \geq 1} a_{n}(f)q^{n}$ be a normalized Hecke eigen newform of even weight $k$ for $\Gamma_0(N)$ and let $K_{f}$ be the Hecke field of $f$.
For a place $\mathfrak{p}$ of $K_{f}$ above $p$, let $V_{f, \mathfrak{p}}$ be a two-dimensional $p$-adic Galois representation of $G_{\Q}$ with coefficients of $K_{f, \mathfrak{p}}$ associated to $f$ constructed by Deligne in \cite{Del69}. 
Here, $K_{f, \fr{p}}$ is the completion of $K_{f}$ at $\fr{p}$. 
We use the convention that the $p$-adic cyclotomic character has Hodge--Tate weight 1, under which $V_{f,\mathfrak{p}}$ has Hodge--Tate weights $1-k$ and 0.
By the Poincar\'e duality, it is known that $V_{\fr{p}} \coloneq V_{f, \fr{p}}(k/2)$ has a non-degenerate skew-symmetric $G_{\Q}$-equivariant pairing
\begin{align}
  [~, ~]:V_{\fr{p}} \times V_{\fr{p}} \to K_{f, \fr{p}}(1).
\end{align}
Assume that $p \nmid N$.
Therefore, $V_{\fr{p}}$ is a crystalline representation at $p$, and satisfies \textbf{(Hyp 1)-(Hyp 3)} by Scholl \cite{Sch90}
(cf. \cite[Section 8.3]{Nek93}).
We recall the structure of $D_{\dR}(V_{\fr{p}})$.
Let $\omega_{f}$ be the element of $D_{\dR}(V_{\fr{p}})$ corresponding to a differential form attached to $f$ by the comparison theorem.
It is known that the de Rham filtration of $D_{\dR}(V_{\fr{p}})$ is given by
\begin{align}
  \Fil^{i} D_{\dR}(V_{\fr{p}}) = \begin{cases}
    D_{\dR}(V_{\fr{p}}) & \left(i \leq - \frac{k}{2} \right), \\
    K_{f, \fr{p}} \omega_{f} & \left(1-\frac{k}{2} \leq i \leq \frac{k}{2} - 1 \right), \\
    0 & \left(i \geq \frac{k}{2} \right).
  \end{cases}  \label{eq:filtration}
\end{align}
The action of $\varphi$ on $D_{\cris}(V_{\fr{p}})$ satisfies $\varphi^{2} -a_{p}(f)p^{-k/2}\varphi + p^{-1} =0$.
Let $\alpha$, $\beta$ be the roots of $X^{2} - a_{p}(f)p^{-k/2}X + p^{-1}$.
Let $L$ be a finite extension of $K_{f, \fr{p}}$ containing $\alpha$ and $\beta$.
We define splittings $N_{\alpha}$ and $N_{\beta}$ in the same way as in the previous section.
Thus, we can define the $p$-adic height pairing 
\begin{align}
  H^{1}_{f}(\Q, V_{\fr{p}}) \times H^{1}_{f}(\Q, V_{\fr{p}}) \to L,  
\end{align}
for a non-trivial continuous homomorphism $l_{\Q}:\A_{\Q}^{\times}/\Q^{\times} \to \Q_{p}$.
Let $[~, ~]_{\dR}$ be the de Rham pairing induced by the above pairing on $V_{\fr{p}}$ as in the previous section.
Applying Theorem \ref{thm:MainTheorem} and Corollary \ref{cor:non-triviality} to $V_{\fr{p}}$, we obtain the following.

\begin{cor}\label{cor:differ_modular}
  Assume that $f$ is non-ordinary at $p$, namely, we assume that $a_{p}(f)$ is not a $p$-adic unit. 
  \begin{enumerate}
      \item For $a, d \in H^{1}_f(\Q, V_{\fr{p}})$, we have
      \begin{align}
        &\langle d, a \rangle_{l_{\Q}, N_{\alpha}} - \langle d, a \rangle_{l_{\Q}, N_{\beta}} \\
        &= (\alpha - \beta) (l_{\Q, p} \circ \exp_{K_{f, \fr{p}}(1)}^{\BK}) \left([\varphi\omega_{f}, \omega_{f}]_{\dR}^{-1} \cdot \log^{\BK}_{V_{\fr{p}}}(d_p)(\omega_f) \cdot \log^{\BK}_{V_{\fr{p}}}(a_p)(\omega_f) \right).
      \end{align}
      In particular, we have
      \begin{align}
        \langle a,a \rangle_{l_{\Q},N_{\alpha}} - \langle a,a \rangle_{l_{\Q},N_{\beta}} = (\alpha - \beta) (l_{\Q, p} \circ \exp_{K_{f, \fr{p}}(1)}^{\BK}) \left([\varphi\omega_{f}, \omega_{f}]_{\dR}^{-1} \cdot \log^{\BK}_{V_{\fr{p}}}(a_{p})(\omega_{f})^{2}\right)
      \end{align}
      for $a \in H^{1}_{f}(\Q, V_{\fr{p}})$.

      \item We also assume that $\alpha\neq\beta$ and that the localization map $H^{1}_{f}(\Q, V_{\fr{p}}) \to H^{1}_{f}(\Q_{p}, V_{\fr{p}})$ at $p$ is non-zero.
      Then, either $\langle ~ , ~ \rangle_{l_{\Q}^{c}, N_{\alpha}}$ or $\langle ~ , ~ \rangle_{l_{\Q}^{c}, N_{\beta}}$ is non-trivial. 
      In particular, if the polynomial $X^2-a_{p}(f)p^{-k/2}X+p^{-1}$ is irreducible over $K_{f, \fr{p}}$, then $\langle ~ , ~ \rangle_{l_{\Q}^{c},N_{\alpha}}$ and $\langle ~ , ~ \rangle_{l_{\Q}^{c}, N_{\beta}}$ are non-trivial.
  \end{enumerate}
\end{cor}

\begin{proof}
  It is sufficient to check that $V_{\fr{p}}$ satisfies the hypothesis (\textbf{Hyp 4}) in order to apply Theorem \ref{thm:MainTheorem} and Corollary \ref{cor:non-triviality}.
  To check it, we show that $\varphi \omega_{f} \notin D_{\dR}^{+}(V_{\fr{p}})$. 
  Assume that $\varphi \omega_{f} \in D_{\dR}^{+}(V_{\fr{p}})$.
  Then, $D_{\dR}^{+}(V_{\fr{p}})$ is a $\varphi$-stable submodule of $D_{\cris}(V_{\fr{p}}) = D_{\dR}(V_{\fr{p}})$.
  By \eqref{eq:filtration}, the Hodge number $t_{H}(D_{\dR}^{+}(V_{\fr{p}}))$ is $k/2 - 1$.
  Since the eigenvalue of $\varphi$ on $D_{\dR}^{+}(V_{\fr{p}})$ is $\alpha$ or $\beta$, the Newton number $t_{N}(D_{\dR}^{+}(V_{\fr{p}}))$ is $\ord_{p}(\alpha)$ or $\ord_{p}(\beta)$, where $\ord_p$ denotes the $p$-adic valuation normalized by $\ord_p(p)=1$.
  However, using the hypothesis that $a_{p}(f)$ is not a $p$-adic unit and considering the Newton polygon of $X^{2}-a_{p}(f)p^{-k/2}X +p^{-1}$, we see that $\ord_{p}(\alpha)$ and $\ord_{p}(\beta)$ are less than $k/2 - 1 = t_{H}(D_{\dR}^{+}(V_{\fr{p}}))$.
  This contradicts the fact that $D_{\cris}(V_{\fr{p}})$ is a weakly admissible filtered $\varphi$-module.
\end{proof}

\bibliographystyle{plain}
\bibliography{literature}

\end{document}